\documentclass[10pt]{amsart}
\usepackage{amsmath,amscd}
\usepackage{amsbsy}
\usepackage{amssymb}
\usepackage{amscd,amsthm}

\usepackage[all,cmtip]{xy}

\newtheorem{thm}{Theorem}\numberwithin{thm}{section}
\newtheorem{lem}[thm]{Lemma}
\newtheorem{prop}[thm]{Proposition}
\newtheorem{cor}[thm]{Corollary}
\newtheorem{exam}[thm]{Example}
\newtheorem{rema}[thm]{Remark}

\newtheorem{defi}[thm]{Definition}

\newtheorem*{thm2}{Theorem}

\newtheorem*{prop2}{Proposition}

\begin{document}
\begin{center}
\huge{Absolutely split locally free sheaves on proper $k$-schemes and Brauer--Severi varieties}\\[1cm]
\end{center}
\begin{center}

\large{Sa$\mathrm{\check{s}}$a Novakovi$\mathrm{\acute{c}}$}\\[0,5cm]
{\small April 2018}\\[0,5cm]
\end{center}
{\small \textbf{Abstract}. 
We classify absolutely split vector bundles on proper $k$-schemes. More precise, we prove that the closed points of the Picard scheme are in one-to-one correspondence with indecomposable absolutely split vector bundles. Furthermore, we apply the obtained results to study the geometry of (generalized) Brauer--Severi varieties.
\begin{center}
\tableofcontents
\end{center}
\section{Introduction}
The goal of this paper is to classify absolutely split vector bundles on arbitrary proper $k$-schemes $X$. An \emph{absolutely split} vector bundle on a proper $k$-scheme $X$ is a locally free sheaf of finite rank that splits as a direct sum of invertible sheaves after base change to the algebraic closure $\bar{k}$. For simplicity, we call such sheaves $AS$-bundles. It is a classical theorem of Grothendieck [16] that the $AS$-bundles on $\mathbb{P}^1$ are just the locally free sheaves. Biswas and Nagaraj [7], [8], [9] and the author [27], [28] classified $AS$-bundles on certain Brauer--Severi varieties. To be precise, $AS$-bundles were completely classified in dimension one and for Brauer--Severi varieties over $\mathbb{R}$. Since Brauer--Severi varieties are twisted forms of projective spaces, the results in \emph{loc.cit.} are  generalizations of Grothendieck's theorem for $\mathbb{P}^1$. In order to study $AS$-bundles on arbitrary Brauer--Severi varieties $X$, we consider the general case where $X$ is an arbitrary proper $k$-scheme. Our first main result is the following:
\begin{thm2}[Theorem 4.6]
Let $X$ be a proper $k$-scheme with $H^0(X,\mathcal{O}_X)=k$. Then the closed points of the Picard scheme $\mathrm{Pic}_{X/k}$ are in one-to-one correspondence with isomorphism classes of indecomposable $AS$-bundles on $X$.
\end{thm2}
Notice that Theorem 4.6 is a direct generalization of the fact that the twisted Picard group equals the set of $k$-rational points in the Picard scheme (see \cite{KO}, (7.4)). Since $AS$-bundles are given by the direct sum of indecomposable $AS$-bundles, we have a one-to-one correspondence of finite closed subschemes of $\mathrm{Pic}_{X/k}$ and isomorphism classes of $AS$-bundles. In the case where the proper scheme $X\otimes_k \bar{k}$ has infinite cyclic Picard group, we obtain:
\begin{thm2}[Theorem 5.1]
Let $X$ be a proper $k$-scheme with $\mathrm{Pic}(X\otimes_k\bar{k})\simeq \mathbb{Z}$ and period $r$ and let $\mathcal{J}$ be a generator of $\mathrm{Pic}(X)$. Denote by $\mathcal{L}$ the generator of $\mathrm{Pic}(X\otimes_k\bar{k})\simeq \mathbb{Z}$ satisfying $\mathcal{J}\otimes_k \bar{k}\simeq \mathcal{L}^{\otimes r}$. Assume there is an indecomposable pure bundle $\mathcal{M}_{\mathcal{L}}$ of type $\mathcal{L}$. Then all indecomposable $AS$-bundles $\mathcal{E}$ are of the form
\begin{eqnarray*}
\mathcal{J}^{\otimes a}\otimes \mathcal{M}_{\mathcal{L}^{\otimes j}}
\end{eqnarray*}
with unique $a\in\mathbb{Z}$ and $0\leq j\leq r-1$.
\end{thm2}
See Remark 5.2 for an argument why the indecomposable $AS$-bundles in Theorem 5.1 do not depend on the choice of the generator of $\mathrm{Pic}(X)$. Notice that for coherent sheaves on proper $k$-schemes the Krull--Schmidt Theorem holds (see [5]). Therefore, any coherent sheaf can uniquely  be decomposed as a direct sum of indecomposables (up to isomorphism and permutation). It is easy to see that all $AS$-bundles on proper $k$-schemes are obtained as a direct sum of the indecomposable $AS$-bundles.

For a more deeper understanding of the $AS$-bundles it is important to determine the ranks of $\mathcal{M}_{\mathcal{L}^{\otimes j}}$. If $H^0(X,\mathcal{O}_X)=k$, an easy computation shows that the endomorphism algebra $\mathrm{End}(\mathcal{M}_{\mathcal{L}})$ is central simple. One therefore has the notion of the index of $\mathrm{End}(\mathcal{M}_{\mathcal{L}})$. Under the assumption that the Picard group is infinite cyclic, we have the following concrete description of the ranks of $\mathcal{M}_{\mathcal{L}^{\otimes j}}$. 
\begin{prop2}[Proposition 5.3]
Let $X$ be a proper $k$-scheme with $H^0(X,\mathcal{O}_X)=k$ and $\mathrm{Pic}(X\otimes_k\bar{k})\simeq \mathbb{Z}$. Denote by $\mathcal{L}$ the generator of $\mathrm{Pic}(X\otimes_k\bar{k})$ and assume there is an indecomposable vector bundle $\mathcal{M}_{\mathcal{L}}$ satisfying $\mathcal{M}_{\mathcal{L}}\otimes_k \bar{k}\simeq \mathcal{L}^{\oplus d}$. Then for all $j\in\mathbb{Z}$ one has
\begin{eqnarray*}
\mathrm{rk}(\mathcal{M}_{\mathcal{L}^{\otimes j}})=\mathrm{ind}(\mathrm{End}(\mathcal{M}_{\mathcal{L}})^{\otimes j}).
\end{eqnarray*}
\end{prop2}
Applying the above results to the case where $X$ is a (generalized) Brauer--Severi variety we obtain a complete classification of $AS$-bundles on arbitrary (generalized) Brauer--Severi varieties (see Theorem 6.5 and Corollary 6.6). This gives the results of Biswas and Nagaraj [7], [8], [9] and of the author [27], [28] as corollaries. We also consider the sequence of natural numbers $(d_j)_{j\in\mathbb{Z}}$ with $d_j=\mathrm{rk}(\mathcal{M}_{\mathcal{L}^{\otimes j}})$. In fact, for a Brauer--Severi variety of period $p$ it is enough to consider the $p+1$-tupel $(d_0,d_1,...,d_{p-1},d_p)$ (see Proposition 5.4). This $p+1$-tupel will be called the \emph{AS-type}. We determine the $AS$-type and study the relation between the $AS$-types of Brauer equivalent and birational Brauer--Severi varieties. Moreover, we will show that the $AS$-type is a birational invariant (see Proposition 6.12). Finally, as consequences of the Horrocks criterion and a result of Ottaviani [30], we obtain cohomological criteria for a vector bundle on a (generalized) Brauer--Severi variety to be a $AS$-bundle (see Theorem 6.14 and Theorem 7.7).\\

{\small \textbf{Acknowledgement}. This paper is an improvement of the first two chapters of my Ph.D. thesis which was supervised by Stefan Schr\"oer whom I would like to thank for a lot of comments and fruitful discussions.\\

{\small \textbf{Conventions}. Throughout this work $k$ denotes an arbitrary ground field and $k^{sep}$ and $\bar{k}$ a separable respectively algebraic closure. Furthermore, any locally free sheaf is assumed to be of finite rank and will be called vector bundle.

\section{Generalities on Brauer--Severi varieties and simple algebras}
Throughout the paper $k$ denotes an arbitrary field, unless stated otherwise. We recall the basics of Brauer--Severi varieties and central simple algebras and refer to [4], [15], [33] and [34] for details. For the more general notions of Brauer--Severi schemes and Azumaya algebras we refer the reader to [17] and [18]. A \emph{Brauer--Severi variety} of dimension $n$ is a scheme $X$ of finite type over $k$ such that $X\otimes_k L\simeq \mathbb{P}^n$ for a finite field extension $k\subset L$. Such a field extension $k\subset L$ is called \emph{splitting field} of $X$. Clearly, the algebraic closure $\bar{k}$ is a splitting field for any Brauer--Severi variety. One can show that a Brauer--Severi variety always splits over a finite separable field extension of $k$ (see [15], Corollary 5.1.4). By embedding the finite separable splitting field into its Galois closure, a Brauer--Severi variety splits over a finite Galois extension of the base field $k$ (see [15], Corollary 5.1.5). It follows from [19], IV, Chapter II, Theorem 2.7.1 that $X$ is projective, integral and smooth over $k$. There is a well-known one-to-one correspondence between Brauer--Severi varieties and central simple $k$-algebras. Recall that an associative $k$-algebra $A$ is called \emph{central simple} if it is an associative finite-dimensional $k$-algebra that has no two-sided ideals other than $0$ and $A$ and if its center equals $k$. If the algebra $A$ is a division algebra, it is called \emph{central division algebra}. Such central simple $k$-algebras can be characterized by the following well-known fact (see [15], Theorem 2.2.1): $A$ is a central simple $k$-algebra if and only if there is a finite field extension $k\subset L$ such that $A\otimes_k L \simeq M_n(L)$ if and only if $A\otimes_k \bar{k}\simeq M_n(\bar{k})$.

The \emph{degree} of a central simple algebra $A$ is now defined to be $\mathrm{deg}(A):=\sqrt{\mathrm{dim}_k A}$. According to the Wedderburn Theorem, for any central simple $k$-algebra $A$ there is an integer $n>0$ and a division algebra $D$ such that $A\simeq M_n(D)$. The division algebra $D$ is also central and unique up to isomorphism. Now the degree of the unique central division algebra $D$ is called the \emph{index} of $A$ and is denoted by $\mathrm{ind}(A)$. It can be shown that the index is the smallest among the degrees of finite separable splitting fields of $A$ (see [15], Corollary 4.5.9). Two central simple $k$-algebras $A\simeq M_n(D)$ and $B\simeq M_m(D')$ are called \emph{Brauer equivalent} if $D\simeq D'$. Brauer equivalence is indeed an equivalence relation and one defines the \emph{Brauer group} $\mathrm{Br}(k)$ of a field $k$ as the group whose elements are equivalence classes of central simple $k$-algebras and group operation being the tensor product. It is an abelian group with inverse of the equivalence class of $A$ given by the equivalence class of $A^{op}$. The neutral element is the equivalence class of $k$. The order of a central simple $k$-algebra $A$ in $\mathrm{Br}(k)$ is called the \emph{period} of $A$ and is denoted by $\mathrm{per}(A)$. It can be shown that the period divides the index and that both, period and index, have the same prime factors (see [15], Proposition 4.5.13). Denoting by $\mathrm{BS}_n(k)$ the set of all isomorphism classes of Brauer--Severi varieties of dimension $n$ and by $\mathrm{CSA}_{n+1}(k)$ the set of all isomorphism classes of central simple $k$-algebras of degree $n+1$, there is a canonical identification 
\begin{center}
$\mathrm{CSA}_{n+1}(k)=\mathrm{BS}_n(k)$ 
\end{center}
via non-commutative Galois cohomology (see [4], [15] and [35] for details). Hence any $n$-dimensional Brauer--Severi variety $X$ corresponds to a central simple $k$-algebra of degree $n+1$. In view of the one-to-one correspondence between Brauer--Severi varieties and central simple algebras one can also speak about the period of a Brauer--severi variety $X$. It is defined to be the period of the corresponding central simple $k$-algebra $A$.\\ 

We also need some facts concerning simple and semisimple rings in general (see [2], [11] for details). Recall that  for a ring $R$ with unity, a $R$-module $M$ is called \emph{simple} if $M$ has no non-trivial submodules. A $R$-module $M$ is called \emph{semisimple} if $M$ is isomorphic to the direct sum of simple modules (see [2], Chapter 3, \text{\S} 9). Note that a ring $R$ is called simple (semisimple) if it is a simple (semisimple) left module over itself. Clearly, a simple ring $R$ is of course semisimple. The \emph{Jacobson radical} of a ring $R$ is by definition the intersection of all maximal left ideals in $R$ and is denoted by $\mathrm{rad}(R)$. For semisimple rings one has the following well-known characterization: Assume $R$ is left artinian, then $R$ is semisimple if and only if $\mathrm{rad}(R)=0$ (see [2], Proposition 15.16). In particular $R/\mathrm{rad}(R)$ is semisimple. Since central simple $k$-algebras $A$ are isomorphic to $M_n(D)$, they are simple in the above sense (see [2], \text{\S} 13, Example 13.1) and one therefore has $\mathrm{rad}(A)=0$. This fact is needed in the next section (see Proposition 3.2).

\section{Pure vector bundles}
In this section we study a certain class of vector bundles. We follow the work of Arason, Elman and Jacob [3] and call them \emph{pure}. 

\begin{defi}
\textnormal{A vector bundle $\mathcal{E}$ on a proper $k$-scheme $X$ is called \emph{pure of type} $\mathcal{W}$ if there is an indecomposable vector bundle $\mathcal{W}$ on $X\otimes_k \bar{k}$ such that $\mathcal{E}\otimes_k {\bar{k}}\simeq \mathcal{W}^{\oplus m}$. If the indecomposable vector bundle $\mathcal{W}$ is invertible, we say $\mathcal{E}$ is \emph{pure of rank one}.} 
\end{defi}
\begin{exam}
\textnormal{Let $X$ be a $n$-dimensional Brauer--Severi variety over $k$. We consider the Euler sequence on $X\otimes_k\bar{k}\simeq\mathbb{P}^n$ (see [21], II Theorem 8.13):}
\begin{eqnarray}
0\longrightarrow \Omega^1_{\mathbb{P}^{n}}\longrightarrow \mathcal{O}_{\mathbb{P}^n}(-1)^{\oplus (n+1)}\longrightarrow \mathcal{O}_{\mathbb{P}^{n}}\longrightarrow 0.
\end{eqnarray}
\textnormal{This short exact sequence does not split since $\mathcal{O}_{\mathbb{P}^n}(-1)$ has no global sections. Applying $\mathrm{Hom}(\mathcal{O}_{\mathbb{P}^n},-)$ to this sequence gives $\mathrm{Ext}^1(\mathcal{O}_{\mathbb{P}^n},\Omega^1_{\mathbb{P}^n})\simeq \bar{k}$. Therefore, the middle term of the Euler sequence is unique up to isomorphism. Furthermore, since $\mathcal{O}_X$ and $\Omega^1_{X/k}$ exist on $X$ and $\mathrm{Ext}^1(\mathcal{O}_X,\Omega^1_{X/k})\simeq k$, there is also a non-split short exact sequence on $X$} 
\begin{eqnarray}
0\longrightarrow \Omega^1_{X/k}\longrightarrow \mathcal{V}\longrightarrow \mathcal{O}_X\longrightarrow 0,
\end{eqnarray}
\textnormal{where the vector bundle $\mathcal{V}$ is unique up to isomorphism. After base change to $\bar{k}$ one gets back the sequence (1) on $\mathbb{P}^{n}$. Therefore $\mathcal{V}\otimes_k\bar{k}\simeq \mathcal{O}_{\mathbb{P}^n}(-1)^{\oplus (n+1)}$. This shows that $\mathcal{V}$ is pure of type $\mathcal{O}_{\mathbb{P}^n}(-1)$.}
\end{exam}

The vector bundle $\mathcal{V}$ from sequence (2) has an interesting property. Consider the endomorphism algebra $\mathrm{End}(\mathcal{V})$ which is a finite-dimensional associative $k$-algebra. After base change to $\bar{k}$, we find $\mathrm{End}(\mathcal{V})\otimes_k\bar{k}\simeq \mathrm{End}(\mathcal{O}_{\mathbb{P}^n}(-1)^{\oplus (n+1)})\simeq \mathrm{End}(\mathcal{O}_{\mathbb{P}^n}^{\oplus (n+1)})\simeq M_{n+1}(\bar{k})$. We see that $\mathrm{End}(\mathcal{V})$ is a central simple $k$-algebra. More generally, we make the following observation.
\begin{prop}
Let $X$ be a geometrically integral and proper $k$-scheme and $\mathcal{E}$ pure of rank one. Then $\mathrm{End}(\mathcal{E})$ is a central simple $k$-algebra. If furthermore $\mathcal{E}$ is indecomposable, $\mathrm{End}(\mathcal{E})$ is a division algebra. 
\end{prop} 
\begin{proof}
By definition there exists an invertible sheaf $\mathcal{L}$ on $X\otimes_k\bar{k}$ such that $\mathcal{E}\otimes_k {\bar{k}}\simeq\bigoplus^n_{i=0}\mathcal{L}$. Let us denote $X_{\bar{k}}$ for the scheme $X\otimes_k\bar{k}$. We have $\mathcal{E}\otimes_k\bar{k}\otimes \mathcal{L}^{\vee}\simeq \mathcal{O}_{X_{\bar{k}}}^{\oplus (n+1)}$. Since $X$ is supposed to be proper, the endomorphism algebra $\mathrm{End}(\mathcal{E})$ is a finite-dimensional $k$-algebra. Furthermore, since $X$ is geometrically integer we have $\mathrm{End}(\mathcal{O}_{X_{\bar{k}}})\simeq H^0(X_{\bar{k}},\mathcal{O}_{X_{\bar{k}}})\simeq \bar{k}$ and therefore we obtain
\begin{eqnarray*}
 \mathrm{End}(\mathcal{E})\otimes_k\bar{k} & \simeq & \mathrm{End}(\bigoplus^n_{i=0}\mathcal{L})\\
  & \simeq & \mathrm{Hom}(\mathcal{L}\otimes(\bigoplus^n_{i=0}\mathcal{O}_{X_{\bar{k}}}),\mathcal{L}\otimes(\bigoplus^n_{i=0}\mathcal{O}_{X_{\bar{k}}}) )\\
	& \simeq & \mathrm{Hom}((\bigoplus^n_{i=0}\mathcal{O}_{X_{\bar{k}}}),\mathcal{L}^{\vee}\otimes\mathcal{L}\otimes(\bigoplus^n_{i=0}\mathcal{O}_{X_{\bar{k}}}) )\\
	& \simeq & \mathrm{End}(\mathcal{O}_{X_{\bar{k}}}^{\oplus (n+1)})\\
	& \simeq & M_{n+1}(\bar{k}).
 \end{eqnarray*} 
This implies that $\mathrm{End}(\mathcal{E})$ has to be a central simple $k$-algebra. Now suppose that $\mathcal{E}$ is indecomposable. According to [3], p.1324 on a proper $k$-scheme $X$ a vector bundle $\mathcal{E}$ is indecomposable if and only if $\mathrm{End}(\mathcal{E})/\mathrm{rad}(\mathrm{End}(\mathcal{E}))$ is a division algebra over $k$. Since $\mathrm{End}(\mathcal{E})$ is a central simple $k$-algebra we have $\mathrm{rad}(\mathrm{End}(\mathcal{E}))=0$. Hence $\mathrm{End}(\mathcal{E})$ is a central division algebra.
\end{proof}
Note that the rank of a pure sheaf as in Proposition 3.3 is equal to the degree of the central simple $k$-algebra $\mathrm{End}(\mathcal{E})$. Indeed, the above proof shows 
\begin{eqnarray}
\mathrm{rk}(\mathcal{E})=n+1=\mathrm{deg}(\mathrm{End}(\mathcal{E})).
\end{eqnarray}
Now let $X$ be a geometrically integral and proper $k$-scheme. Assume $\mathcal{E}$ is pure of rank one. Suppose furthermore, we are given an indecomposable direct summand $\mathcal{F}$ of $\mathcal{E}$. We want to understand the relation between the central simple $k$-algebras $\mathrm{End}(\mathcal{E})$ and $\mathrm{End}(\mathcal{F})$. The following result is very useful and will be needed later on quite frequently. 
\begin{prop}
Let $X$ be a proper $k$-scheme and $\mathcal{F}$ and $\mathcal{G}$ two coherent sheaves. If $\mathcal{F}\otimes_k \bar{k}\simeq \mathcal{G}\otimes_k \bar{k}$, then $\mathcal{F}$ is isomorphic to $\mathcal{G}$.
\end{prop}
\begin{proof}
Wiegand [37], Lemma 2.3 proved this in the case $X$ is projective over $k$, but the proof remains valid for proper $k$-schemes. For convenience to the reader we reproduce the proof and follow exactly the lines of the proof given in [37]. Since the extension $k\subset \bar{k}$ is a direct limit of finite extensions, it suffices to prove the statement for finite field extensions. Now suppose $k\subset L$ is a finite field extension and $\pi:X\otimes_k L\rightarrow X$ the projection. Choose a basis $\{\alpha_1,...,\alpha_d\}$ for $L$ over $k$. By assumption we have $\pi^*\mathcal{F}\simeq \pi^*\mathcal{G}$. For the coherent sheaf $\mathcal{A}=\pi_*\pi^*\mathcal{F}$ we have over any affine open set $U\subset X$, $\mathcal{A}(U)=\mathcal{F}(U)\otimes_k L$ and there is a unique $\mathcal{O}_X(U)$-module isomorphism from $\mathcal{A}(U)$ to $\mathcal{F}(U)^{\oplus d}$, assigning $m\otimes\alpha_i$ to $(0,0,...,m,...,0)$, where $m$ is located at the $i^{th}$ entry. This yields $\pi_*\pi^*\mathcal{F}\simeq \mathcal{F}^{\oplus d}$ and obviously the same holds for $\pi_*\pi^*\mathcal{G}$. Hence $\pi_*\pi^*\mathcal{F}\simeq \mathcal{F}^{\oplus d}\simeq \pi_*\pi^*\mathcal{G}\simeq \mathcal{G}^{\oplus d}$ and we conclude from Krull--Schmidt Theorem that $\mathcal{F}\simeq \mathcal{G}$.  
\end{proof}
Notice that the statement of Proposition 3.4 also holds if one considers the base change to the separable closure $k^{sep}$ instead of the base change to $\bar{k}$ (see [3], p.1325). A simple consequence of Proposition 3.4 is the following:
\begin{prop}
Let $X$ be a proper and geometrically integral $k$-scheme. If $\mathcal{E}$ and $\mathcal{E}'$ are two indecomposable pure sheaves of same type, then $\mathcal{E}\simeq \mathcal{E}'$. 
\end{prop}
We are now able to understand the relation between the central simple $k$-algebras $\mathrm{End}(\mathcal{E})$ and $\mathrm{End}(\mathcal{F})$ from above.
\begin{prop}
Let $X$ be a proper and geometrically integral $k$-scheme. Let $\mathcal{E}$ be indecomposable and pure of rank one and $\mathcal{F}$ an indecomposable direct summand of $\mathcal{E}$. Then $\mathrm{End}(\mathcal{E})\simeq M_n(\mathrm{End}(\mathcal{F}))$ and hence $\mathrm{End}(\mathcal{E})$ and $\mathrm{End}(\mathcal{F})$ are Brauer equivalent. 
\end{prop}
\begin{proof}
Since $X$ is geometrically integral and proper, Proposition 3.3 yields that $\mathrm{End}(\mathcal{E})$ is a central simple $k$-algebra. Furthermore, since $X$ is proper we can decompose $\mathcal{E}$ as a direct sum of indecomposable vector bundles according to the Krull--Schmidt Theorem. Now let $\mathcal{E}=\bigoplus^n_{i=1}\mathcal{E}_i$ be the Krull--Schmidt decomposition of $\mathcal{E}$. Since $\mathcal{E}$ is pure of rank one, there is an invertible sheaf $\mathcal{L}$ such that $\mathcal{E}\otimes_k\bar{k}\simeq \bigoplus^r_{j=1}\mathcal{L}$. In view of $\mathcal{E}\otimes_k\bar{k}\simeq \bigoplus^n_{i=1}(\mathcal{E}_i\otimes_k\bar{k})$, we have an isomorphism
\begin{eqnarray*}
\bigoplus^r_{j=1}\mathcal{L}\simeq \mathcal{E}\otimes_k\bar{k}\simeq \bigoplus^n_{i=1}(\mathcal{E}_i\otimes_k\bar{k}).
\end{eqnarray*} Therefore, by the Krull--Schmidt Theorem for vector bundles on $X\otimes_k\bar{k}$, we obtain that $\mathcal{E}_i$ is also pure of type $\mathcal{L}$. Since all the bundles $\mathcal{E}_i$ are indecomposable, Proposition 3.5 yields that they are all isomorphic. Hence $\mathcal{E}\simeq\bigoplus^n_{i=1}\mathcal{E}_1$ and thus $\mathcal{F}\simeq \mathcal{E}_1$ by Krull--Schmidt Theorem. By Proposition 3.3, the endomorphism algebra $\mathrm{End}(\mathcal{F})$ is a central division algebra and therefore $\mathrm{End}(\mathcal{E})\simeq\mathrm{End}(\bigoplus^n_{i=1}\mathcal{F})\simeq M_n(\mathrm{End}(\mathcal{F}))$. This shows that $\mathrm{End}(\mathcal{F})$ and $\mathrm{End}(\mathcal{E})$ are Brauer equivalent.
\end{proof}
\begin{rema}
\textnormal{The proof of the above proposition shows that all the indecomposable direct summands in the Krull--Schmidt decomposition of a vector bundle which is pure of rank one are isomorphic. In particular, $\mathrm{End}(\mathcal{F})$ is isomorphic to the unique central division algebra $D$ with $M_n(D)\simeq \mathrm{End}(\mathcal{E})$}.
\end{rema}
\section{AS-bundles on proper $k$-schemes}
In this section we prove the first main result of the present paper. We start with the main definition. 
\begin{defi}
\textnormal{Let $X$ be a $k$-scheme. A vector bundle $\mathcal{E}$ on $X$ is called \emph{absolutely split} (\emph{separably split}) if it splits after base change as a direct sum of invertible sheaves on $X\otimes_k \bar{k}$ (resp. $X\otimes_k k^{sep}$). For an absolutely split locally free sheaf we shortly write \emph{AS-bundle}.}
\end{defi}
\begin{prop}
Let $X$ be a proper $k$-scheme and $\mathcal{E}$ a vector bundle on $X$. Then $\mathcal{E}$ is absolutely split if and only if it is separably split.
\end{prop}
\begin{proof}
If $\mathcal{E}$ is separably split, then it is also absolutely split. For the other implication, assume $\mathcal{E}$ is not separably split. Then after base change to $k^{sep}$ we obtain from the Krull--Schmidt decomposition 
\begin{eqnarray*}
\mathcal{E}\otimes_k k^{sep}\simeq \bigoplus^r_{i=1}\mathcal{E}_i 
\end{eqnarray*} that at least one of the indecomposable $\mathcal{E}_i$ on $X\otimes_k k^{sep}$ has rank $>1$. According to [31], Proposition 3.1, the locally free sheaves $\mathcal{E}_i$ remain indecomposable after base change to $\bar{k}$. Therefore, there exists at least one $\mathcal{E}_i\otimes_{k^{sep}}\bar{k}$ with $\mathrm{rk}(\mathcal{E}_i\otimes_{k^{sep}}\bar{k})>1$. This implies that
\begin{eqnarray*}
\mathcal{E}\otimes_k \bar{k}\simeq (\mathcal{E}\otimes_k k^{sep})\otimes_{k^{sep}}\bar{k}\simeq \bigoplus^r_{i=1}(\mathcal{E}_i\otimes_{k^{sep}}\bar{k}) 
\end{eqnarray*} is not absolutely split. This contradicts the assumption and completes the proof.
\end{proof}
For the proof of Theorem 4.5 we need the following generalization  of [3], Proposition 3.4.
\begin{prop}
Let $X$ be a smooth projective variety over $k$. Let $\mathcal{E}$ be an indecomposable vector bundle on $X\otimes_k k^s$ and suppose $\{\mathcal{E}_1,...,\mathcal{E}_r\}$ is the $\mathrm{Gal}(k^s|k)$-orbit of $\mathcal{E}$. Then there is an up to isomorphism unique indecomposable vector bundle $\mathcal{F}$ on $X$ such that $\mathcal{F}\otimes_k k^s\simeq \bigoplus^r_{i=1}\mathcal{E}^{\oplus d}_i$ for a unique positive integer $d>0$.
\end{prop}
\begin{proof}
The vector bundle $\mathcal{V}:=\mathcal{E}_1\oplus \mathcal{E}_2\oplus...\oplus \mathcal{E}_r$ is by assumption Galois invariant. Since $\mathcal{E}$ is indecomposable, it follows that any $\mathcal{E}_i$ is indecomposable, too. To get our assertion, we proceed as in the proof of Proposition 3.4 in [3] to obtain a vector bundle $\mathcal{M}$ on $X$ satisfying $\mathcal{M}\otimes_k k^s\simeq \mathcal{V}^{\oplus m}=\bigoplus^r_{i=1}\mathcal{E}^{\oplus m}_i$ for a suitable positive integer $m>0$. Take any direct summand $\mathcal{W}$ of $\mathcal{M}$ and observe that $\mathcal{W}\otimes_k k^s\simeq \mathcal{V}^{\oplus r}$ for some positive integer $r\leq m$. In fact this follows from the assumption that $\{\mathcal{E}_1,...,\mathcal{E}_r\}$ is the $\mathrm{Gal}(k^s|k)$-orbit of the indecomposable bundle $\mathcal{E}$ and because $\bigoplus^r_{i=1}\mathcal{E}^{\oplus m}_i$ is the Krull--Schmidt decomposition of $\mathcal{M}\otimes_k k^s$. Now choose among the direct summands of $\mathcal{M}$ a bundle $\mathcal{F}$ with smallest rank. Denote this rank by $d$. Now one proceeds as in the proof of Proposition 3.4 in [3] to conclude with Krull--Schmidt Theorem and Proposition 3.4 that $\mathcal{F}$ is unique up to isomorphism.
\end{proof}
For $G=\mathrm{Gal}(k^{sep}|k)$ we denote by $\mathrm{Pic}^G(X\otimes_k k^{sep})\subset \mathrm{Pic}(X\otimes_k k^{sep})$ the subgroup of isomorphism classes of $G$-invariant invertible sheaves. An immediate consequence of Proposition 4.3 is:  
\begin{cor}
Let $X$ be a proper $k$-scheme. For all $\mathcal{L}\in \mathrm{Pic}^G(X\otimes_k k^{sep})$ there is an up to isomorphism unique indecomposable $\mathcal{M}_{\mathcal{L}}$ on $X$ such that $\mathcal{M}_{\mathcal{L}}\otimes_k\bar{k}\simeq \mathcal{L}^{\oplus n}$.
\end{cor}
More general, we have:
\begin{thm}
Let $X$ be a proper $k$-scheme and $\mathcal{M}$ an indecomposable separably split vector bundle. Then there is a unique positive integer $d>0$ such that $\mathcal{M}\otimes_k k^{sep}\simeq \mathcal{L}^{\oplus d}_1\oplus\cdots\oplus \mathcal{L}^{\oplus d}_m$, where $\{\mathcal{L}_1,...,\mathcal{L}_l\}$ is the $\mathrm{Gal}(k^{sep}|k)$-orbit of a unique $\mathcal{L}\in \mathrm{Pic}^G(X\otimes_k k^{sep})$. 
\end{thm}
\begin{proof}
Let $\mathcal{M}$ be an indecomposable separably split vector bundle. By definition we have 
\begin{eqnarray*}
\mathcal{M}\otimes_k k^{sep}\simeq \bigoplus_{i=1}^n\mathcal{L}_i^{\oplus m_i},
\end{eqnarray*} for suitable invertible sheaves $\mathcal{L}_i$ on $X\otimes_k k^{sep}$. Since $\mathcal{M}\otimes_k k^{sep}$ is $\mathrm{Gal}(k^{sep}|k)$ invariant, the set $\{\mathcal{L}_1,...,\mathcal{L}_n\}$ decomposes as the disjoint union of $\mathrm{Gal}(k^{sep}|k)$-orbits. Let us denote by $\mathbb{B}_l=\{\mathcal{M}(l)_1,...,\mathcal{M}(l)_{s_l}\}$ the respective orbits. Without loss of generality, $\mathbb{B}_1=\{\mathcal{L}_1,...,\mathcal{L}_r\}=\{\sigma^*(\mathcal{L}_1)\mid \sigma\in\mathrm{Gal}(k^{sep}|k)\}$. We claim that $m_1=m_2=...=m_r$. Assume by contradiction that there is a $j>0$ such that $m_1\neq m_j$, $j=1,...,r$. Again, without loss of generality $j=2$. Now there is a $\sigma\in\mathrm{Gal}(k^{sep}|k)$ with $\sigma^*(\mathcal{L}_1)\simeq \mathcal{L}_2$. We then must have
\begin{eqnarray*}
\sigma^*(\mathcal{L}_1)^{\oplus m_1}\oplus\sigma^*(\mathcal{L}_2)^{\oplus m_2}\oplus\cdots\oplus \sigma^*(\mathcal{L}_r)^{\oplus m_r}\simeq \mathcal{L}^{\oplus m_1}_1\oplus \mathcal{L}^{\oplus m_2}_2\oplus\cdots \oplus\mathcal{L}^{\oplus m_r}_r. 
\end{eqnarray*}  
It follows from the fact that all $\mathcal{L}_i$ are indecomposable and Krull--Schmidt Theorem that $m_1=m_2$. This contradicts the assumption $m_1\neq m_2$. The corresponding statement applies for any $\mathbb{B}_l$. 

Now by Proposition 4.3 it follows that for any $\mathbb{B}_l$ there is a unique $d_l$ and a unique $\mathcal{F}_l$ such that 
\begin{eqnarray*}
\mathcal{F}_l\otimes_k k^{sep}\simeq \mathcal{M}(l)^{\oplus d_l}_1\oplus\cdots\oplus\mathcal{M}(l)^{\oplus d_l}_{s_l}.  
\end{eqnarray*}
 We denote by $d$ the least common multiple of all the $d_l$. By definition, there are integers $n_l\in\mathbb{Z}$ such that $n_l\cdot d_l=d$. We find
\begin{eqnarray*}
(\mathcal{M}^{\oplus d})\otimes_k k^{sep}& \simeq &\bigoplus^r_{l=1}\bigoplus^{s_l}_{j=1}(\mathcal{M}(l)^{\oplus b_l}_j)^{\oplus d}\\
& \simeq & \bigoplus^r_{l=1}\bigoplus^{s_l}_{j=1}(\mathcal{M}(l)^{\oplus b_l\cdot d}_j)\\
& \simeq & \bigoplus^r_{l=1}\bigoplus^{s_l}_{j=1}(\mathcal{M}(l)^{\oplus d_l}_j)^{\oplus n_l\cdot b_l}\\
& \simeq & (\bigoplus^r_{l=1}\mathcal{F}^{\oplus n_l\cdot b_l}_l)\otimes_k k^{sep}.
\end{eqnarray*}
Now Proposition 3.3 shows
\begin{eqnarray*}
\mathcal{M}^{\oplus d}\simeq \bigoplus^r_{l=1}\mathcal{F}^{\oplus n_l\cdot b_l}_l.
\end{eqnarray*}
The Krull--Schmidt Theorem implies that $\mathcal{M}$ must be isomorphic to some $\mathcal{F}_l$. Hence the assertion.
\end{proof}

Theorem 4.5 has a interesting geometric consequence which is the content of the main theorem of this section. To prove this theorem, we first have to look more closely at the Hochschild--Serre spectral sequence for Galois coverings (see [26] for details). In general, for a proper $k$-scheme $X$ with $H^0(X,\mathcal{O}_X)=k$ one has the four-term exact sequence for the Galois covering $\mathrm{Spec}(k^{sep})\rightarrow \mathrm{Spec}(k)$ with $X'=X\otimes_k k^{sep}$:\\

$0\longrightarrow H^1(G,H^0(X'_{et},\mathbb{G}_m))\longrightarrow H^1(X_{et},\mathbb{G}_m)\longrightarrow H^0(G,H^1(X'_{et},\mathbb{G}_m))\longrightarrow$\\

$\longrightarrow H^2(G,H^0(X'_{et},\mathbb{G}_m))$.\\

Assuming $H^0(X,\mathcal{O}_X)=k$ one has $H^2(G,H^0(X'_{et},\mathbb{G}_m))=\mathrm{Br}(k)$ and the above exact sequence becomes
\begin{eqnarray}
0\longrightarrow \mathrm{Pic}(X)\longrightarrow \mathrm{Pic}^G(X')\longrightarrow \mathrm{Br}(k).
\end{eqnarray} 
In order to get a more geometric interpretation of the group $\mathrm{Pic}^G(X')$ we recall the basics of the Picard scheme. The main references for the Picard scheme are [19], [20] and [23].\\
For a scheme $X$, the Picard group $\mathrm{Pic}(X)$ is the same as $H^1(X,\mathcal{O}^*_X)$ (see [21], p.224). This group is also called the absolute Picard group. To get some relative version of this group we fix a $S$-scheme $X$ with structural morphism $f:X\rightarrow S$. Now for a $S$-scheme $T$ one has the following base change diagram
\begin{displaymath}
\begin{xy}
  \xymatrix{
      X_T=X\times_S T\ar[r]^{} \ar[d]_{f_T}    &   X\ar[d]^{f}                   \\
      T \ar[r]^{}             &   S             
  }
\end{xy}
\end{displaymath}\\
and one can form the presheaf $T\mapsto H^1(X_T,\mathcal{O}^*_{X_T})$ whose associated sheaf is $\mathbb{R}^1f_{T_*}\mathcal{O}^*_{X_T}$ (see [21], Proposition 8.1, p.250). In the Zariski topology one then defines the relative Picard functor simply as $\mathrm{Pic}_{(X/S)(Zar)}(T)=H^0(T,\mathbb{R}^1f_{T_*}\mathcal{O}^*_{X_T})$. One now can imitate this for the fppf-topology and define $\mathrm{Pic}_{(X/S)(fppf)}(T)=H^0(T,\mathbb{R}^1f_{T_*}\mathbb{G}_m)$. Again the Leray spectral sequence 
\begin{eqnarray*}
E^{pq}_2=H^p(T, \mathbb{R}^q f_{T_*}\mathcal{F}_{X_T})\Longrightarrow H^{p+q}(X_T,\mathcal{F}_{X_T}),
\end{eqnarray*} provides us with the exact sequence of low degree
\begin{eqnarray*}
0\longrightarrow H^1(T,f_{T_*}\mathcal{F})\longrightarrow H^1(X_{T},\mathcal{F})\longrightarrow H^0(T,\mathbb{R}^1f_{T_*}\mathcal{F})\longrightarrow H^2(T,f_{T_*}\mathcal{F}).
\end{eqnarray*}
The cohomology groups occurring in the exact sequence are meant with respect to the fppf-topology. Now for the sheaf $\mathbb{G}_m$ in the fppf-topology, that is also a sheaf in the \'etale-topology, the above exact sequence becomes: 
\begin{center}
$0\longrightarrow H^1(T,f_{T_*}\mathbb{G}_m)\longrightarrow H^1(X_{T},\mathbb{G}_m)\longrightarrow H^0(T,\mathbb{R}^1f_{T_*}\mathbb{G}_m)\longrightarrow H^2(T,f_{T_*}\mathbb{G}_m)$.
\end{center}
Assuming $f_*\mathcal{O}_X\simeq \mathcal{O}_S$ one obtains $f_{T_*}\mathbb{G}_m=\mathbb{G}_m$ and hence $H^1(T,f_{T_*}\mathbb{G}_m)=H^1(T,\mathbb{G}_m)$, what by [20], p.190-216 implies $H^1(T,f_{T_*}\mathbb{G}_m)=\mathrm{Pic}(T)$. Thus the above exact sequence becomes
\begin{center}
$0\longrightarrow \mathrm{Pic}(T)\longrightarrow \mathrm{Pic}(X_T)\longrightarrow H^0(T,\mathbb{R}^1f_{T_*}\mathbb{G}_m)\longrightarrow H^2(T,f_{T_*}\mathbb{G}_m)$.
\end{center} 
One defines the \emph{relative Picard functor} $\mathrm{Pic}_{(X/S)(et)}(T)$ in the \'etale-topology simply as $H^0(T,\mathbb{R}^1f_{T_*}\mathbb{G}_m)$ since the above exact sequence also exists in this topology. It is a very delicate problem under what kind of assumptions the Picard functor is representable in dependence of the given topology. We refer to [20] and [23] for details. The main theorem of Grothendieck about the Picard functor is that under the assumption that $f:X\rightarrow S$ is projective and flat and its geometric fibers are integral, the Picard functor  $\mathrm{Pic}_{(X/S) (et)}$ is representable (see [23], Theorem 4.8). The object that represents the Picard functor is a separated scheme locally of finite type over $S$, called the \emph{Picard scheme} of $X$ and is denoted by $\mathrm{Pic}_{X/S}$. In particular, for a proper $k$-scheme $X$ the Picard functor $\mathrm{Pic}_{(X/k) (et)}$ is also representable (see [23], Theorem 4.18.2, Corollary 4.18.3 or [20] p.236 ff.). Specializing further, the above exact sequence becomes for $S=\mathrm{Spec}(k)=T$ and $H^0(X,\mathcal{O}_X)=k$:
\begin{eqnarray*}
0\longrightarrow \mathrm{Pic}(\mathrm{Spec}(k))\longrightarrow \mathrm{Pic}(X)\longrightarrow \mathrm{Pic}_{(X/k)(fppf)}(k)\longrightarrow \mathrm{Br}(k).
\end{eqnarray*} As mentioned above, if $X$ is supposed to be proper over $S=\mathrm{Spec}(k)=T$ and $H^0(X,\mathcal{O}_X)=k$, the above sequence also holds in the \'etale-topology and we consider the Picard functor $H^0(T,\mathbb{R}^1f_{T_*}\mathbb{G}_m)$ as given in this topology. This gives us the following exact sequence:
\begin{eqnarray}
0\longrightarrow \mathrm{Pic}(\mathrm{Spec}(k))\longrightarrow \mathrm{Pic}(X)\longrightarrow \mathrm{Pic}_{(X/k)(et)}(k)\longrightarrow \mathrm{Br}(k).
\end{eqnarray}
This means that we can represent elements of $\mathrm{Pic}_{(X/k)(et)}(k)$ by invertible sheaves on $X$ that become zero in the Brauer group $\mathrm{Br}(k)$. If $X$ has a $k$-rational point, one can show that $\mathrm{Pic}(X)=\mathrm{Pic}_{(X/k)(et)}(k)$ and therefore the $k$-rational points of $\mathrm{Pic}_{X/k}$ are in one-to-one correspondence with invertible sheaves on $X$. The question arises what happens if $X$ does not admit a $k$-rational point. It is a well known fact that elements of $\mathrm{Pic}_{(X/k)(et)}(k)$ are in one-to-one correspondence with so called twisted line bundles on $X$ (see \cite{KO}, (7.4)). A generalization of this fact is the following theorem. 

\begin{thm}
Let $X$ be a proper $k$-scheme with $H^0(X,\mathcal{O}_X)=k$. Then the closed points of $\mathrm{Pic}_{X/k}$ are in one-to-one correspondence with isomorphism classes of indecomposable $AS$-bundles on $X$.
\end{thm}
\begin{proof}
We denote by $\mathrm{AS}_X$ the set of isomorphism classes of indecomposable $AS$-bundles on $X$. In view of Proposition 4.2, we can restrict ourselves to consider separably split bundles. Let $y\in \mathrm{Pic}_{X/k}$ be a closed point. Then $k(y)$ is a finite extension of $k$. The closed point $y$ corresponds to a $\mathrm{Gal}(k^{sep}|k)$-orbit $\{y_1,...,y_r\}$ of points in $\mathrm{Pic}_{X/k}(k^{sep})$. Recall the sequence (5) for $X^{sep}:=X\otimes_k k^{sep}$
\begin{eqnarray*}
0\longrightarrow \mathrm{Pic}(\mathrm{Spec}(k^{sep}))\longrightarrow \mathrm{Pic}(X^{sep})\longrightarrow \mathrm{Pic}_{(X^{sep}/k^{sep})(et)}(k^{sep})\longrightarrow \mathrm{Br}(k^{sep}).
\end{eqnarray*}
Since $\mathrm{Pic}_{X/k}(k^{sep})=\mathrm{Pic}_{(X^{sep}/k^{sep})}(k^{sep})$ and $\mathrm{Br}(k^{sep})=0$, we see that $\mathrm{Pic}(X^{sep})$ is isomorphic to $\mathrm{Pic}_{X/k}(k^{sep})$. Notice that this isomorphism is compatible with the action of $\mathrm{Gal}(k^{sep}|k)$. So for any $\mathrm{Gal}(k^{sep}|k)$-orbit $\{y_1,...,y_r\}$ of points in $\mathrm{Pic}_{X/k}(k^{sep})$, there is a unique $\mathrm{Gal}(k^{sep}|k)$-orbit $\{\mathcal{L}_1,...,\mathcal{L}_r\}$ of line bundles in $\mathrm{Pic}(X^{sep})$. By Proposition 4.2 and Proposition 4.3, there is a uniquely determined indecomposable $AS$-bundle $\mathcal{M}_y$ satisfying 
\begin{eqnarray*}
\mathcal{M}_y\otimes_k k^{sep}\simeq \mathcal{L}^{\oplus d_y}_1\oplus\cdots\oplus \mathcal{L}^{\oplus d_y}_r.
\end{eqnarray*}
So there is a well defined map
\begin{eqnarray*}
\phi\colon \{\textnormal{closed points in}\ \mathrm{Pic}_{X/k}\}\longrightarrow \mathrm{AS}_X,
\end{eqnarray*}
assigning to each closed point $y$ the $AS$-bundle $\mathcal{M}_y$ as described above. From Theorem 4.5 it follows that $\phi$ is surjective. It remains to show that $\phi$ is injective. For this, let $y$ and $y'$ be closed points and $\mathcal{M}_y$ and $\mathcal{M}_{y'}$ the corresponding indecomposable $AS$-bundles. Assume $\mathcal{M}_y\simeq \mathcal{M}_{y'}$. From Theorem 4.5 we conclude 
\begin{eqnarray*}
\mathcal{M}_y\otimes_k k^{sep}\simeq \mathcal{L}^{\oplus d}_1\oplus\cdots\oplus \mathcal{L}^{\oplus d}_r\simeq \mathcal{N}^{\oplus d'}_1\oplus\cdots\oplus \mathcal{N}^{\oplus d'}_s\simeq \mathcal{M}_{y'}\otimes_k k^{sep}
\end{eqnarray*}
The Krull--Schmidt Theorem implies $r=s$ and $d=d'$. Moreover, we have $\mathcal{L}_i\simeq \mathcal{N}_j$. Therefore, the Galois orbits $\{y_1,...,y_r\}$ and $\{y'_1,...,y'_s\}$ corresponding to the orbits $\{\mathcal{L}_1,...,\mathcal{L}_r\}$ and $\{\mathcal{N}_1,...,\mathcal{N}_s\}$ are the same. By descent, this gives $y=y'$. 
\end{proof}

\section{AS-bundles on proper $k$-schemes with cyclic Picard group}
In this section we prove the second main theorem  of the present paper. The results of this section will be applied in Section 6 to study Brauer--Severi varieties.\\

We start by fixing some notation and stating some facts. From Proposition 3.4 we conclude that $\mathrm{Pic}(X)$ is a subgroup of $\mathrm{Pic}(X\otimes_k\bar{k})$. In particular, if $\mathrm{Pic}(X\otimes_k\bar{k})\simeq \mathbb{Z}$, we have $\mathrm{Pic}(X)\simeq r\mathbb{Z}$. We call the integer $|r|$ the \emph{period} of $X$. Let us fix a generator $\mathcal{J}$ of $\mathrm{Pic}(X)$.  
Now let $\mathcal{L}\in \mathrm{Pic}(X\otimes_k\bar{k})$ be the generator satisfying $\mathcal{J}\otimes_k \bar{k}\simeq \mathcal{L}^{\otimes |r|}$. Assume there is a pure bundle $\mathcal{M}$ of type $\mathcal{L}\in \mathrm{Pic}(X\otimes_k\bar{k})$. We know from Proposition 3.5 that the bundle $\mathcal{M}$ is unique up to isomorphism. We use the notation $\mathcal{M}_{\mathcal{L}}$. It is easy to see that for any invertible sheaf $\mathcal{L}^{\otimes j}\in \mathrm{Pic}(X\otimes_k\bar{k})$ there is an indecomposable pure bundle of type $\mathcal{L}^{\otimes j}$. This is due to the following fact: Let $s=\mathrm{rk}(\mathcal{M}_{\mathcal{L}})$ and consider $(\mathcal{L}^{\oplus s})^{\otimes j}\simeq (\mathcal{L}^{\otimes j})^{\oplus s^j}$. From this one obtains $\mathcal{M}_{\mathcal{L}}^{\otimes j}\otimes_k\bar{k}\simeq (\mathcal{L}^{\oplus s})^{\otimes j}\simeq (\mathcal{L}^{\otimes j})^{\oplus s^j}$. Considering the Krull--Schmidt decomposition of $\mathcal{M}_{\mathcal{L}}^{\otimes j}$ and taking into account that all indecomposable direct summands are isomorphic (see proof of Proposition 3.6 and Remark 3.7), we get a up to isomorphism unique indecomposable locally free sheaf $\mathcal{M}_{\mathcal{L}^{\otimes j}}$ such that $\mathcal{M}_{\mathcal{L}^{\otimes j}}\otimes_k\bar{k}\simeq (\mathcal{L}^{\otimes j})^{\oplus s_j}$, where $s_j$ is the rank of $\mathcal{M}_{\mathcal{L}^{\otimes j}}$.\\ 

With the above notation we can prove the following result.
\begin{thm}
Let $X$ be a proper $k$-scheme with $\mathrm{Pic}(X\otimes_k\bar{k})\simeq \mathbb{Z}$ and period $r$ and let $\mathcal{J}$ be a generator of $\mathrm{Pic}(X)$. Denote by $\mathcal{L}$ the generator of $\mathrm{Pic}(X\otimes_k\bar{k})\simeq \mathbb{Z}$ satisfying $\mathcal{J}\otimes_k \bar{k}\simeq \mathcal{L}^{\otimes r}$. Assume there is an indecomposable pure bundle $\mathcal{M}_{\mathcal{L}}$ of type $\mathcal{L}$. Then all indecomposable $AS$-bundles $\mathcal{E}$ are of the form
\begin{eqnarray*}
\mathcal{J}^{\otimes a}\otimes \mathcal{M}_{\mathcal{L}^{\otimes j}}
\end{eqnarray*}
with unique $a\in\mathbb{Z}$ and $0\leq j\leq r-1$.
\end{thm}
\begin{proof}
Let $\mathcal{E}$ be an arbitrary, not necessarily indecomposable $AS$-bundle and $\pi:X\otimes_k \bar{k}\rightarrow X$ the projection. By assumption, there is an indecomposable pure vector bundle $\mathcal{M}_{\mathcal{L}}$ of type $\mathcal{L}$. We have shown above that there exist up to isomorphism unique indecomposable pure vector bundles of type $\mathcal{L}^{\otimes j}$ for all $j\in \mathbb{Z}$. We denote them by $\mathcal{M}_{\mathcal{L}^{\otimes j}}$. Now consider $\mathcal{M}_{\mathcal{L}^{\otimes j}}$ only for $j=0,...,r-1$. Let $d=\mathrm{lcm}(\mathrm{rk}(\mathcal{M}_{\mathcal{O}_{X_{\bar{k}}}}),\mathrm{rk}(\mathcal{M}_{\mathcal{L}}),...,\mathrm{rk}(\mathcal{M}_{\mathcal{L}^{\otimes (r-1)}}))$ be the least common multiple and consider the vector bundle $\pi^*(\mathcal{E}^{\oplus d})$. Since $\mathcal{E}$ is an $AS$-bundle, the vector bundle $\mathcal{E}^{\oplus d}$ is an $AS$-bundle, too. Therefore $\pi^*(\mathcal{E}^{\oplus d})$ decomposes into a direct sum of invertible sheaves and we find after reordering $\mathrm{mod}$ $r$ that $\pi^*(\mathcal{E}^{\oplus d})$ is isomorphic to
\begin{eqnarray*}
\left(\bigoplus^{s_0}_{i=0}(\mathcal{L}^{\otimes {a_{i_{0}}\cdot r}})^{\oplus d}\right)\oplus \left(\bigoplus^{s_1}_{i=0}(\mathcal{L}^{\otimes {a_{i_{1}}\cdot r}+1})^{\oplus d}\right)\oplus\cdots\\
\oplus \left(\bigoplus^{s_{r-1}}_{i=0}((\mathcal{L}^{\otimes {a_{i_{r-1}}\cdot r}+(r-1)}))^{\oplus d}\right).
\end{eqnarray*}
By definition of $d$, there are $h_j$ such that $h_j\cdot\mathrm{rk}(\mathcal{M}_{\mathcal{L}^{\otimes j}})=d$ for $0\leq j\leq r-1$. Furthermore, the sheaves $\mathcal{M}_{\mathcal{L}^{\otimes j}}$ satisfy $\pi^*\mathcal{M}_{\mathcal{L}^{\otimes j}}\simeq(\mathcal{L}^{\otimes j})^{\oplus d_j}$, where $d_j=\mathrm{rk}(\mathcal{M}_{\mathcal{L}^{\otimes j}})$. Now for the direct summands $(\mathcal{L}^{\otimes {a_{i_j}}\cdot r + j})^{\oplus d}$ we have 
\begin{eqnarray*}
\left(\mathcal{L}^{\otimes a_{i_j}\cdot r + j}\right)^{\oplus d}=\left(\left(\mathcal{L}^{\otimes a_{i_j}\cdot r + j}\right)^{\oplus d_j}\right)^{\oplus h_j}.
\end{eqnarray*} By considering the vector bundle $(\mathcal{J}^{\otimes a_{i_j}}\otimes \mathcal{M}_{\mathcal{L}^{\otimes j}})^{\oplus h_j}$ on $X$, we find
\begin{eqnarray*}
\pi^*\left(\mathcal{J}^{\otimes a_{i_j}}\otimes \mathcal{M}_{\mathcal{L}^{\otimes j}}\right)^{\oplus h_j}\simeq \left(\mathcal{L}^{\otimes {a_{i_j}}\cdot r + j}\right)^{\oplus d}.
\end{eqnarray*} For the vector bundle 
\begin{eqnarray*}
\mathcal{R}=\left(\bigoplus^{r_0}_{i=0}(\mathcal{J}^{\otimes {a_{i_0}}})^{\oplus d}\right)\oplus \left(\bigoplus^{r_1}_{i=0}(\mathcal{J}^{\otimes {a_{i_1}}})\otimes\mathcal{M}_{\mathcal{L}}^{\oplus h_1}\right)\oplus\cdots\\
\oplus \left(\bigoplus^{r_{p-1}}_{i=0}(\mathcal{J}^{\otimes {a_{i_{p-1}}}})\otimes\mathcal{M}_{\mathcal{L}^{\otimes (r-1)}}^{\oplus h_{r-1}}\right)
\end{eqnarray*} we therefore have $\pi^*\mathcal{R}\simeq \pi^*(\mathcal{E}^{\oplus d})$. Applying Proposition 3.4 shows that $\mathcal{E}^{\oplus d}$ is isomorphic to $\mathcal{R}$. Because Krull--Schmidt Theorem holds for vector bundles on $X$, we conclude that $\mathcal{E}$ is isomorphic to the direct sum vector bundles of the form $\mathcal{J}^{\otimes a}\otimes \mathcal{M}_{\mathcal{L}^{\otimes j}}$ with unique $a\in\mathbb{Z}$ and $0\leq j\leq r-1$. Furthermore, since all these bundles are indecomposable by definition, we finally get that all the indecomposable $AS$-bundles are of the form $\mathcal{J}^{\otimes a}\otimes \mathcal{M}_{\mathcal{L}^{\otimes j}}$ with unique $a\in\mathbb{Z}$ and $0\leq j\leq r-1$. This completes the proof.
\end{proof} 
\begin{rema}
\textnormal{If we take the other generator $\mathcal{J}^{\vee}$ of $\mathrm{Pic}(X)$, we must take $\mathcal{L}^{\vee}$ as the generator of $\mathrm{Pic}(X\otimes_k\bar{k})$ to have $\mathcal{J}^{\vee}\otimes_k\bar{k}\simeq (\mathcal{L}^{\vee})^{\otimes r}$. Imitating the proof of Theorem 5.1 get that all indecomposable $AS$-bundles are of the form $(\mathcal{J}^{\otimes a}\otimes \mathcal{M}_{\mathcal{L}^{\otimes j}})^{\vee}$ with unique $a\in\mathbb{Z}$ and $-r+1\leq j\leq 0$. Now Proposition 5.4 below shows that these bundles are isomorphic to the indecomposable $AS$-bundles from Theorem 5.1. Therefore, the isomorphism classes of indecomposable $AS$-bundles do not depend on the choice of the generator of $\mathrm{Pic}(X)$.}
\end{rema}
Now we want to determine the ranks of the sheaves $\mathcal{M}_{\mathcal{L}^{\otimes j}}$. If $H^0(X,\mathcal{O}_X)=k$, we obtain from Proposition 3.3 that $\mathrm{End}(\mathcal{M}_{\mathcal{L}})$ is central simple. One therefore has the notion of the index of $\mathrm{End}(\mathcal{M}_{\mathcal{L}})$. 
\begin{prop}
Let $X$ be a proper $k$-scheme with $H^0(X,\mathcal{O}_X)=k$ and $\mathrm{Pic}(X\otimes_k\bar{k})\simeq \mathbb{Z}$. Denote by $\mathcal{L}$ a generator of $\mathrm{Pic}(X\otimes_k\bar{k})$ and assume there is an indecomposable pure vector bundle $\mathcal{M}_{\mathcal{L}}$ of type $\mathcal{L}$. Then for all $j\in\mathbb{Z}$ 
\begin{eqnarray*}
\mathrm{rk}(\mathcal{M}_{\mathcal{L}^{\otimes j}})=\mathrm{ind}(\mathrm{End}(\mathcal{M}_{\mathcal{L}})^{\otimes j}).
\end{eqnarray*}
\end{prop} 
\begin{proof}
Let $\pi:X\otimes_k \bar{k}\rightarrow X$ denote the projection. For the bundle $\mathcal{M}_{\mathcal{L}}^{\otimes j}$ we have
\begin{eqnarray*}
\pi^*\mathcal{M}_{\mathcal{L}}^{\otimes j}\simeq (\mathcal{L}^{\oplus \mathrm{rk}(\mathcal{M}_{\mathcal{L}})})^{\otimes j}\simeq (\mathcal{L}^{\otimes j})^{\oplus \mathrm{rk}(\mathcal{M}_{\mathcal{L}})^j}.
\end{eqnarray*} This shows that $\mathcal{M}_{\mathcal{L}}^{\otimes j}$ is pure of type $\mathcal{L}^{\otimes j}$. In general, $\mathcal{M}_{\mathcal{L}}^{\otimes j}$ need not to be indecomposable. Since $H^0(X,\mathcal{O}_X)=k$, Proposition 3.3 implies that the endomorphism algebra $\mathrm{End}(\mathcal{M}_{\mathcal{L}^{\otimes j}})$ is a central division algebra. Applying the Krull--Schmidt Theorem for $\mathcal{M}_{\mathcal{L}}^{\otimes j}$ on $X$ we get a decomposition 
\begin{eqnarray*}
\mathcal{M}_{\mathcal{L}}^{\otimes j}\simeq \bigoplus_{i=1}^m\mathcal{E}_i,
\end{eqnarray*} where all $\mathcal{E}_i$ are indecomposable. After base change to the algebraic closure we have
\begin{eqnarray*}
(\mathcal{L}^{\oplus \mathrm{rk}(\mathcal{M}_{\mathcal{L}})})^{\otimes j}\simeq (\mathcal{M}_{\mathcal{L}}^{\otimes j})\otimes_k\bar{k}\simeq \bigoplus_{i=1}^m\mathcal{E}_i\otimes_k \bar{k}.
\end{eqnarray*} 
Applying Krull--Schmidt Theorem for any $\mathcal{E}_{i}$ on $X\otimes_{k} \bar{k}$ shows that $\mathcal{E}_{i}$ are pure of type $\mathcal{L}^{\otimes j}$. According to Proposition 3.5 and Remark 3.7, all $\mathcal{E}_i$ are isomorphic to $\mathcal{M}_{\mathcal{L}^{\otimes j}}$. Proposition 3.6 implies that $\mathrm{End}(\mathcal{M}_{\mathcal{L}}^{\otimes j})$ is a matrix algebra over the central division algebra $\mathrm{End}(\mathcal{M}_{\mathcal{L}^{\otimes j}})$. By (3), the rank of $\mathcal{M}_{\mathcal{L}^{\otimes j}}$ equals $\mathrm{deg}(\mathrm{End}(\mathcal{M}_{\mathcal{L}^{\otimes j}}))$. But $\mathrm{deg}(\mathrm{End}(\mathcal{M}_{\mathcal{L}^{\otimes j}}))=\mathrm{ind}(\mathrm{End}(\mathcal{M}_{\mathcal{L}}^{\otimes j}))$. Now it is an easy exercise to show $\mathrm{End}(\mathcal{M}_{\mathcal{L}}^{\otimes j})\simeq \mathrm{End}(\mathcal{M}_{\mathcal{L}})^{\otimes j}$ as central simple algebras. Finally, it follows $\mathrm{rk}(\mathcal{M}_{\mathcal{L}^{\otimes j}})=\mathrm{deg}(\mathrm{End}(\mathcal{M}_{\mathcal{L}^{\otimes j}}))=\mathrm{ind}(\mathrm{End}(\mathcal{M}_{\mathcal{L}})^{\otimes j}).$
\end{proof}
Let $X$ be as in Theorem 5.1 and consider the sequence $(\mathrm{rk}(\mathcal{M}_{\mathcal{L}^{\otimes j}}))_{j\in \mathbb{Z}}$. Since $\mathrm{Pic}(X\otimes_k \bar{k})$ is infinite cyclic one has a periodicity in the above sequence with respect to the period $r$ of $X$.
\begin{prop}
Let $X$ be a proper $k$-scheme of period $r$ with $\mathrm{Pic}(X\otimes_k\bar{k})\simeq \mathbb{Z}$ and $H^0(X,\mathcal{O}_X)=k$ and let $\mathcal{L}$, $\mathcal{J}$ and $\mathcal{M}_{\mathcal{L}^{\otimes j}}$ be as in Theorem 5.1. Then for all $j\in\mathbb{Z}$ the following hold:
\begin{itemize}
      \item[\bf (i)] $\mathcal{M}^{\vee}_{\mathcal{L}^{\otimes j}}\simeq \mathcal{M}_{\mathcal{L}^{\otimes (-j)}}$.
      \item[\bf (ii)] $\mathcal{M}_{\mathcal{L}^{\otimes (j+ar)}}\simeq \mathcal{J}^{\otimes a}\otimes \mathcal{M}_{\mathcal{L}^{\otimes j}}$.
			\end{itemize} In particular one has $\mathrm{rk}(\mathcal{M}_{\mathcal{L}^{\otimes j}})=\mathrm{rk}(\mathcal{M}_{\mathcal{L}^{\otimes (-j)}})$ and $\mathrm{rk}(\mathcal{M}_{\mathcal{L}^{\otimes (j+ar)}})=\mathrm{rk}(\mathcal{M}_{\mathcal{L}^{\otimes j}})$.
\end{prop}
\begin{proof}
For the invertible sheaf $\mathcal{L}$ there is are indecomposable vector bundles $\mathcal{M}_{\mathcal{L}}$ and $\mathcal{M}_{\mathcal{L}^{\vee}}$. Since $\mathcal{M}_{\mathcal{L}}$ is pure of type $\mathcal{L}$, we conclude that $\mathcal{M}^{\vee}_{\mathcal{L}}$ is pure of type $\mathcal{L}^{\vee}$. According to Proposition 3.5, $\mathcal{M}^{\vee}_{\mathcal{L}}$ is isomorphic to $\mathcal{M}_{\mathcal{L}^{\vee}}$. The same argument shows $\mathcal{M}^{\vee}_{\mathcal{L}^{\otimes j}}\simeq \mathcal{M}_{\mathcal{L}^{\otimes (-j)}}$. This proves (i). To prove (ii), we apply the same argument. For this, let $\mathcal{J}$ be the generator of $\mathrm{Pic}(X)$ and note that by the definition of the period of $X$ we have $\mathcal{J}\otimes_k\bar{k}\simeq\mathcal{L}^{\otimes r}$. Now consider the bundle $\mathcal{M}_{\mathcal{L}^{\otimes (j+ar)}}$ and note that it is pure of type $\mathcal{L}^{\otimes (j+ar)}$. The vector bundle $\mathcal{J}^{\otimes a}\otimes \mathcal{M}_{\mathcal{L}^{\otimes j}}$ is indecomposable and also pure of type $\mathcal{L}^{\otimes (j+ar)}$. Again with Proposition 3.5 we conclude $\mathcal{M}_{\mathcal{L}^{\otimes (j+ar)}}\simeq \mathcal{J}^{\otimes a}\otimes \mathcal{M}_{\mathcal{L}^{\otimes j}}$. This completes the proof.
\end{proof}

\begin{rema}
\textnormal{Under the assumption of the above proposition we see that the sequence $(\mathrm{rk}(\mathcal{M}_{\mathcal{L}^{\otimes j}}))_{j\in \mathbb{Z}}$ is completely determined by the tupel $(1,\mathrm{rk}(\mathcal{M}_{\mathcal{L}}),...,\mathrm{rk}(\mathcal{M}_{\mathcal{L}^{\otimes r-1}}),1)$. Note that $\mathrm{rk}(\mathcal{M}_{\mathcal{O}_{X\otimes_k\bar{k}}})=\mathrm{rk}(\mathcal{M}_{\mathcal{L}^{\otimes r}})=1$, since both, the structure sheaf and $\mathcal{L}^{\otimes r}$, descent.}
\end{rema}
In view of Proposition 5.4 and Remark 5.5 we define the \emph{AS-type} of $X$ to be the $r+1$-tupel $(1,\mathrm{rk}(\mathcal{M}_{\mathcal{L}}),...,\mathrm{rk}(\mathcal{M}_{\mathcal{L}^{\otimes r-1}}),1)$. It is interesting to study the relation between the $AS$-types of proper $k$-schemes $X$ and $Y$ which are related by certain morphisms $f:X\rightarrow Y$. Moreover, for the central simple $k$-algebras $\mathrm{End}(\mathcal{M}_{\mathcal{L}})$ one also has the notion of period as invariant and it is reasonable to study which geometric properties of $X$ are reflected by the $AS$-type and the periods of $\mathrm{End}(\mathcal{M}_{\mathcal{L}})$.


\section{AS-bundles on Brauer--Severi varieties}
In this section we describe the structure of $AS$-bundles on Brauer--Severi varieties. Furthermore, we study the $AS$-type and find a relation between the $AS$-types of Brauer equivalent and birational Brauer--Severi varieties.\\

As mentioned in Section 2, the $n$-dimensional Brauer--Severi varieties over $k$ are in one-to-one correspondence with central simple $k$-algebras of degree $n+1$. 
As mentioned above, the period of a Brauer--Severi variety $X$ is the order of the corresponding central simple $k$-algebra $A$ in $\mathrm{Br}(k)$. But there is also a geometric interpretation of the period. Hochschild--Serre spectral sequence for Galois coverings applied to a Brauer--Severi variety $X$ yields the following exact sequence (see [4]):
\begin{center}
$0\longrightarrow \mathrm{Pic}(X)\xrightarrow{deg} \mathbb{Z}\longrightarrow\mathrm{Br}(k)$.
\end{center} 
A Theorem due to Lichtenbaum (see [15], Theorem 5.4.10.) states that the boundary map $\delta:\mathbb{Z}\rightarrow \mathrm{Br}(k)$ is given by sending $1$ to the class of $X$ in $\mathrm{Br}(k)$. Here by class of $X$ we mean that of the corresponding central simple $k$-algebra. It follows that $\mathrm{Pic}(X)$ is identified with some subgroup $r\mathbb{Z}$ of $\mathbb{Z}$. So $|r|$ is the order of $A$ in $\mathrm{Br}(k)$, where $A$ is the central simple $k$-algebra corresponding to $X$. Hence the period of $X$ can also be thought of as the smallest positive integer $s$ such that $\mathcal O_X(s)$ exists in $\mathrm{Pic}(X)$. Therefore, the period of a Brauer--Severi variety equals the period as defined in Section 5. For further investigations we need the following result contained in [33], Theorem 5.5. Denote by $(m,n)$ the greatest common divisor of the natural numbers $m$ and $n$.
\begin{thm}
Let $A$ be a central simple $k$-algebra of index $i$. Then for $r\geq 0$ one has:
\begin{itemize}
      \item[\bf (i)] The index of $A^{\otimes r}$ divides $(\binom{i}{r}, i)$.
      \item[\bf (ii)] Suppose $i$ and $r$ are coprime. Then $A^{\otimes r}$ has index $i$.
			\item[\bf (iii)] Let $e$ be $(r,i)$. Then $A^{\otimes r}$ has index dividing $i/e$.
    \end{itemize}
\end{thm} 
Notice that for Brauer--Severi varieties $X$ over $k$ one has $H^0(X,\mathcal{O}_X)=k$ and $\mathrm{Pic}(X)\simeq \mathbb{Z}$. Since $X\otimes_k \bar{k}\simeq \mathbb{P}^n$, one also has $\mathrm{Pic}(X\otimes_k \bar{k})\simeq \mathbb{Z}$. To apply Theorem 5.1 we have to show that for the generator $\mathcal{O}_{\mathbb{P}^n}(1)\in \mathrm{Pic}(X\otimes_k \bar{k})$ there exists an indecomposable pure vector bundle of type $\mathcal{O}_{\mathbb{P}^n}(1)$. 
\begin{prop}
Let $X$ be a $n$-dimensional Brauer--Severi variety over $k$ and $\mathcal{O}_{\mathbb{P}^n}(1)$ the generator of $\mathrm{Pic}(X\otimes_k \bar{k})\simeq \mathbb{Z}$. Then there exists an indecomposable pure vector bundle $\mathcal{W}_1$ on $X$ of type $\mathcal{O}_{\mathbb{P}^n}(1)$.
\end{prop}
\begin{proof}
We have seen that the vector bundle $\mathcal{V}$ of Example 3.2
satisfies $\pi^*\mathcal{V}\simeq \mathcal{O}_{\mathbb{P}^n}(-1)^{\oplus(n+1)}$, where $\pi$ denotes the projection $\pi:X\otimes_k\bar{k}\rightarrow X$. Hence $\pi^*\mathcal{V}^{\vee}$ is isomorphic to $\mathcal{O}_{\mathbb{P}^n}(1)^{\oplus(n+1)}$ and therefore $\mathcal{V}^{\vee}$ is pure of type $\mathcal{O}_{\mathbb{P}^n}(1)$. Let us denote the vector bundle $\mathcal{V}^{\vee}$ by $\mathcal{W}$. We decompose $\mathcal{W}$ according to the Krull--Schmidt Theorem and write $\mathcal{W}\simeq \bigoplus^{m}_{i=1}\mathcal{W}_{i}$. Since all the $\mathcal{W}_i$ have to be isomorphic (see Remark 3.7), it follows $\mathcal{W}\simeq \mathcal{W}^{\oplus m}_1$, where $\mathcal{W}_1$ is unique up to isomorphism. From Krull--Schmidt Theorem it follows that $\mathcal{W}_1$ is also pure of type $\mathcal{O}_{\mathbb{P}^n}(1)$. This completes the proof.
\end{proof}
We saw that the period $p$ of a Brauer--Severi variety $X$ can be thought of as the smallest positive integer $r$ such that $\mathcal O_X(r)$ exists in $\mathrm{Pic}(X)$. Hence the period defined for Brauer--Severi varieties equals the period defined for arbitrary proper $k$-schemes in Section 5. Now we write $\mathcal{W}_j$ for the vector bundles $\mathcal{M}_{\mathcal{L}^{\otimes j}}$ of Section 5 with $\mathcal{L}=\mathcal{O}_{\mathbb{P}^n}(1)$. Then Theorem 5.1 applies and we obtain:
\begin{cor}
Let $X$ be a Brauer--Severi variety over a field $k$ of period $p$. Then all $AS$-bundles $\mathcal{E}$ are of the form
\begin{eqnarray*}
\mathcal{E}\simeq \bigoplus_{j=0}^{p-1}\left(\bigoplus_{i=0}^{r_j}  \mathcal{O}_X(a_{i_j}p)\otimes \mathcal{W}_j\right)
\end{eqnarray*}
with unique $a_{i_j}\in\mathbb{Z}$ and $r_j> 0$.
\end{cor}
As mentioned in Section 5, for a complete understanding of the $AS$-bundles on Brauer--Severi varieties one has to determine the ranks of the $\mathcal{W}_j$. So consider the sequence of natural numbers $(d_j)_{j\in\mathbb{Z}}$, with $d_j=\mathrm{rk}(\mathcal{W}_j)$. Proposition 5.4 and Remark 5.5 show that we do not have to consider the hole sequence $(d_j)_{j\in\mathbb{Z}}$. Furthermore, we note that $\mathcal{W}_0=\mathcal{O}_X$ and $\mathcal{W}_p=\mathcal{O}_X(p)$, where $p$ is the period of $X$. This implies that $\mathrm{rk}(\mathcal{W}_0)=1=\mathrm{rk}(\mathcal{W}_p)$. 
In fact, for the vector bundle $\mathcal{W}$ of the proof of Proposition 6.2 one can show $\mathrm{End}(\mathcal{W})\simeq A$, where $A$ is the central simple algebra corresponding to $X$ (see [32], p.144 or [36], \text{\S}3, 3.6). Applying Proposition 5.3 implies: 
\begin{cor}
Let $X$ be a $n$-dimensional Brauer--Severi variety over a field $k$ corresponding to a central simple $k$-algebra $A$. Then for every $j\in\mathbb{Z}$ one has
\begin{center}
$\mathrm{rk}(\mathcal{W}_j)=\mathrm{ind}(A^{\otimes j})$.
\end{center}
\end{cor}

\begin{thm}
Let $X$ be a Brauer--Severi variety over $k$ and $A$ the corresponding central simple $k$-algebra of period $p$. Then the $AS$-type of $X$ is given by $d_j=\mathrm{ind}(A^{\otimes j})$ for $0\leq j\leq p$.
\end{thm}
Corollary 6.4 together with Theorem 6.5 now give a complete classification of $AS$-bundles on Brauer--Severi varieties and thus the results of Biswas and Nagaraj [7], [8], [9] and the author [27], [28] as corollary. 
\begin{cor}([28], Theorem 5.1 and [8], Theorem 1.1)
Let $X$ be a $n$-dimensional Brauer--Severi variety over $k$ of index two. Then the $AS$-bundles are of the form
\begin{eqnarray*}
\mathcal{E}\simeq \left(\bigoplus_{i=1}^r \mathcal{O}_X(2a_i)\right)\oplus \left(\bigoplus_{j=1}^s \mathcal{O}_X(2b_j)\otimes \mathcal{W}_1\right)
\end{eqnarray*}
with unique $r,s,a_i$ and $b_j$. The vector bundle $\mathcal{W}_1$ satisfies $\mathcal{W}_1\otimes_k L\simeq \mathcal{O}_{\mathbb{P}^n}(1)^{\oplus 2}$, where $k\subset L$ is a degree two Galois extension that splits $X$.
\end{cor}
\begin{proof}
Since the index of $X$ is two and the period divides the index (see [15], Proposition 4.5.13), we conclude that the period is also two. Note that the period cannot be one, since this would imply that $X$ is the projective space, contradicting the fact that the index of $X$ is two. Hence the $AS$-type of $X$ is $(1,2,1)$ according to Theorem 6.5. Now the assertion follows from Corollary 6.3 and [15], Corollary 4.5.9. 
\end{proof}
Note that the results in [7], [9] and [27] concern the case of non-trivial one-dimensional Brauer--Severi varieties and therefore follow directly from Corollary 6.6 since non-split Brauer--Severi varieties of dimension are of index two.\\

In the case where the Brauer--Severi variety corresponds to a central simple $k$-algebra with period equals index, the $AS$-type can be stated very explicitly. We first need the following facts.
\begin{lem}
Let $A$ be a central simple $k$-algebra of period $p$. Then one has $p/(p,r)=\mathrm{per}(A^{\otimes r})$.
\end{lem}
\begin{prop}
Let $A$ be a central simple $k$-algebra with period $p$ and index $i$. Then  for all $r\geq 0$ one has that $p/(p,r)$ divides $\mathrm{ind}(A^{\otimes r})$ and $\mathrm{ind}(A^{\otimes r})$ divides $i/(i,r)$. In particular one has $p/(p,r)\leq\mathrm{ind}(A^{\otimes r})\leq i/(i,r)$. 
\end{prop}
\begin{proof}
By Lemma 6.7 we have $p/(p,r)=\mathrm{per}(A^{\otimes r})$, and since the period always divides the index (see [15], Proposition 4.5.13) we find that $p/(p,r)$ divides $\mathrm{ind}(A^{\otimes r})$. The second inequality $\mathrm{ind}(A^{\otimes r})\leq i/(i,r)$ and the fact that $\mathrm{ind}(A^{\otimes r})$ divides $i/(i,r)$ is (iii) of Theorem 6.1.
\end{proof}
The last proposition enables us to determine the $AS$-type of Brauer--Severi varieties with same period and index.
\begin{prop}
Let $X$ be a Brauer--Severi variety over $k$ corresponding to a central simple $k$-algebra $A$ such that the period $p$ equals the index $i$. Then the AS-type of $X$ is given by 
\begin{center}
$\mathrm{rk}(\mathcal{W}_j)= p/(p,j)$
\end{center} for $0\leq j\leq p$.
\end{prop}
\begin{proof}
Since the period $p$ equals the index $i$, we conclude from Proposition 6.8 that $p/(p,j)=\mathrm{ind}(A^{\otimes j})$ for $0\leq j\leq p$. The assertion then follows from Corollary 6.4.
\end{proof}
\begin{rema}
\textnormal{The problem for which fields $k$ the period equals the index is called \emph{period-index problem} and is highly non-trivial. For further discussion on this problem we refer the reader to [6] and to the work of de Jong [12]}.
\end{rema}
If $k$ is local or global, any central division algebra over $k$ is cyclic and the period equals the index (see [33], Theorem 10.7). In this case Proposition 6.9 applies and we can determine the $AS$-type . The next observation shows that it is enough to consider minimal linear subvarieties to determine the $AS$-type.
\begin{prop}
Let $X$ and $Y$ be Brauer--Severi varieties over $k$ which are Brauer equivalent. Then $X$ and $Y$ have the same $AS$-type.
\end{prop}
\begin{proof}
Let $A$ be the central simple $k$-algebra corresponding to $X$ and $B$ the central simple $k$-algebra corresponding to $Y$. Since $A$ and $B$ are Brauer equivalent, there is a unique central division algebra $D$ such that $A\simeq M_n(D)$ and $B\simeq M_m(D)$ for suitable $n$ and $m$. Hence $\mathrm{ind}(A^{\otimes j})=\mathrm{ind}(D^{\otimes j})=\mathrm{ind}(B^{\otimes j})$ for all $j\in\mathbb{Z}$. Since Brauer equivalent Brauer--Severi varieties have the same period, Theorem 6.5 yields the assertion.
\end{proof}
Interestingly, it turns out that the $AS$-type is also a birational invariant for Brauer--Severi varieties.
\begin{prop}
Let $X$ and $Y$ be two birational Brauer--Severi varieties. Then they have the same $AS$-type.
\end{prop}
\begin{proof}
Let $A$ be the central simple $k$-algebra corresponding to $X$ and $B$ the algebra corresponding to $Y$. Since $X$ and $Y$ are supposed to be birational, [15], Corollary 5.4.2 implies that $A$ and $B$ generate the same cyclic subgroup in $\mathrm{Br}(k)$. Denoting by $i$ the index of $A$, Proposition 6.8 shows that $\mathrm{ind}(A^{\otimes r})$ divides $i/(i,r)$. Since $i/(i,r)$ divides $i$, we conclude that $\mathrm{ind}(A^{\otimes r})$ divides $i=\mathrm{ind}(A)$. The same holds for $B$ and we see that $\mathrm{ind}(B^{\otimes s})$ divides $\mathrm{ind}(B)$. In what follows we show that $\mathrm{ind}(A^{\otimes r})=\mathrm{ind}(B^{\otimes r})$ for all $r$. Since $A$ and $B$ generate the same cyclic subgroup in $\mathrm{Br}(k)$, we know that $A$ is Brauer equivalent to $B^{\otimes l}$ and $B$ to $A^{\otimes m}$ for suitable $l$ and $m$. Therefore $\mathrm{ind}(A^{\otimes m})=\mathrm{ind}(B)$ divides $\mathrm{ind}(A)$ and $\mathrm{ind}(B^{\otimes l})=\mathrm{ind}(A)$ divides $\mathrm{ind}(B)$. This shows that $\mathrm{ind}(B)$ divides $\mathrm{ind}(A)$ and vice verse. Thus $\mathrm{ind}(B)=\mathrm{ind}(A)$. The same argument applied to $\mathrm{ind}(A^{\otimes r})$ and $\mathrm{ind}(B^{\otimes r})$ yields that $\mathrm{ind}(A^{\otimes rm})=\mathrm{ind}(B^{\otimes r})$ divides $\mathrm{ind}(A^{\otimes r})$ and that $\mathrm{ind}(B^{\otimes rl})=\mathrm{ind}(A^{\otimes r})$ divides $\mathrm{ind}(B^{\otimes r})$. This shows $\mathrm{ind}(A^{\otimes r})=\mathrm{ind}(B^{\otimes r})$ for all $r$. As mentioned above, $A$ and $B$ generate the same cyclic subgroup and therefore have the same period. Now Theorem 6.5 implies that $X$ and $Y$ have the same $AS$-type.
\end{proof}
\begin{rema}
\textnormal{Note that the other implication in Proposition 6.11 and 6.12 does not hold. Indeed, by [15], Theorem 1.4.2, two non-split one-dimensional Brauer--Severi varieties $X$ and $Y$ are birational if and only if they are isomorphic. But according to Theorem 6.5 and Corollary 6.6 any one-dimensional Brauer--Severi variety has the same $AS$-type. So indeed there are Brauer--Severi varieties which are not Brauer equivalent (respectively not birational) but have the same $AS$-type.}
\end{rema}
As a consequence of the Horrocks criterion on $\mathbb{P}^n$ (see [29]) we can state a cohomological criterion for a vector bundle on a Brauer--Severi variety to be a $AS$-bundle.

\begin{thm}[AS-criterion]  
Let $X$ be a $n$-dimensional Brauer--Severi variety over $k$ and period $p$. A vector bundle $\mathcal{E}$ is a $AS$-bundle if and only if for $0<i<n$ one has
\begin{center}
$H^i(X,\mathcal{E}\otimes\mathcal{O}_X(ap)\otimes \mathcal{W}_j)=0$
\end{center}
for every $a\in\mathbb{Z}$ and every $0\leq j\leq p-1$.
\end{thm}
We end up this section mentioning that the duals of the indecomposable $AS$-bundles $\mathcal{W}_j$ with $j\geq 0$ on a Brauer--Severi variety $X$ generate the Grothendieck group $K_0(X)$. This follows directely from [32], Section 8, Theorem 4.1 or [22], Theorem 3.1. 
 
\section{AS-bundles on generalized Brauer--Severi varieties}
In this section we will see that nearly all results stated in the last section for Brauer--Severi varieties also hold for generalized Brauer--Severi varieties.\\

Let $A$ be a central simple $k$-algebra of degree $n$ and $1\leq d\leq n$. Now consider the subset of $\mathrm{Grass}_k(d\cdot n,A)$ consisting of those subspaces of $A$ that are left ideals $L$ of dimension $d\cdot n$. This subset of $\mathrm{Grass}_k(d\cdot n,A)$ can be given a structure of a projective scheme over $k$, defined by the relations stating that the $L$ are left ideals (see [10], p.100 for details). This scheme is denoted by $\mathrm{BS}(d,A)$ and is called \emph{generalized Brauer--Severi variety}. It is a closed subscheme of the Grassmannian $\mathrm{Grass}_k(dn,A)$. For a field extension $k\subset E$ one has $\mathrm{BS}(d,A\otimes_k E)\simeq\mathrm{BS}(d,A)\otimes_k E$ (see [10]). As in the case of Brauer--Severi varieties, there is always a finite separable and therefore a finite Galois extension $k\subset K$ such that $\mathrm{BS}(d,A)\otimes_k K\simeq \mathrm{Grass}_K(d,n)$. Clearly, the generalized Brauer--Severi variety becomes isomorphic to the Grassmannian after base change to the algebraic closure $\bar{k}$. Note that for $d=1$ one obtains the usual Brauer--Severi variety as defined in Section 2. Unfortunately, there is no one-to-one correspondence between generalized Brauer--Severi varieties and central simple $k$-algebras (see [10], Theorem 1). Nonetheless, it is easy to classify all $AS$-bundles on generalized Brauer--Severi varieties (see Theorem 7. and Proposition 7.2 below). Recall the tautological exact sequence on $X=\mathrm{Grass}_k(d,n)$ 
\begin{eqnarray*}
0\longrightarrow \mathcal{S}\longrightarrow \mathcal{O}_X^{\oplus n}\longrightarrow \mathcal{Q}\longrightarrow 0
\end{eqnarray*}
with \emph{tautological sheaf} $\mathcal{S}$ which is a vector bundle of rank $d$. On the generalized Brauer--Severi variety $\mathrm{BS}(d,A)$ one has also a tautological short exact sequence 
\begin{eqnarray*}
0\longrightarrow \mathcal{I}\longrightarrow \mathcal{O}_{\mathrm{BS}(d,A)}^{\oplus n^2}\longrightarrow \mathcal{R}\longrightarrow 0.
\end{eqnarray*}
This short becomes after base change $\mathrm{BS}(d,A)\otimes_k L\simeq\mathrm{Grass}_{L}(d,n)$ to some splitting field $L$ of $A$ (see [25], Section 4):
\begin{eqnarray*}
0\longrightarrow \mathcal{S}^{\oplus n}\longrightarrow \mathcal{O}_{\mathrm{Grass}_{L}(d,n)}^{\oplus n^2}\longrightarrow \mathcal{Q}^{\oplus n}\longrightarrow 0.
\end{eqnarray*} 
The vector bundle $\mathcal{I}$ is called \emph{tautological sheaf} on $Y=\mathrm{BS}(d,A)$. Recall that for $\mathrm{Grass}(d,n)$ the Picard group coincides with its first Chow group (see [14], Example 15.3.6). By Proposition 14.6.6 in [14] we have $\mathrm{Pic}(\mathrm{Grass}(d,n))\simeq \mathbb{Z}$, where the ample generator can be taken to be $\mathcal{L}=\mathrm{det}(\mathcal{S}^{\vee})$. Proposition 3.3 now implies $\mathrm{Pic}(\mathrm{BS}(d,A))\simeq \mathbb{Z}$. In order to apply Theorem 5.1, we have to show that there is a up to isomorphism unique pure vector bundle of type $\mathcal{L}$ on $Y$. For this, let $\Sigma^{\lambda}$ be the Schur functor associated with the partition $\lambda=(\lambda_1,\lambda_2,...,\lambda_{d})$, where $0\leq \lambda_i\leq n-d$ (see [1], [13] for details). Levine, Srinivas and Weyman [25], Section 4 proved that for the tautological sheaf $\mathcal{S}$ on $\mathrm{Grass}(d,n)$ the vector bundles $\Sigma^{\lambda}(\mathcal{S})^{\oplus n\cdot|\lambda|}$ descent to vector bundles $\mathcal{N}_{\lambda}$ on $\mathrm{BS}(d,A)$. Note that these vector bundles are unique up to isomorphism by Proposition 3.3. In particular, for the ample generator $\mathcal{L}=\mathrm{det}(\mathcal{S}^{\vee})$ of $\mathrm{Pic}(\mathrm{Grass}(d,n))$ there is a up to isomorphism unique indecomposable vector bundle $\mathcal{W}_{\mathcal{L}}$ satisfying $\mathcal{W}_{\mathcal{L}}\otimes_k k^{sep}\simeq \mathcal{L}^{\oplus \mathrm{rk}(\mathcal{W}_{\mathcal{L}})}$. As explained in Section 5, for all $\mathcal{L}^{\otimes j}\in \mathrm{Pic}(\mathrm{Grass}(d,n))$ there exist up to isomorphism unique indecomposable vector bundles $\mathcal{W}_{\mathcal{L}^{\otimes j}}$ with $\mathcal{W}_{\mathcal{L}^{\otimes j}}\otimes_k k^{sep}\simeq (\mathcal{L}^{\otimes j})^{\oplus \mathrm{rk}(\mathcal{W}_{\mathcal{L}^{\otimes j}})}$. Let $\mathcal{M}$ be the generator of $\mathrm{Pic}(\mathrm{BS}(d,A))$ satisfying $\mathcal{M}\otimes_k k^{sep}\simeq \mathcal{L}^{\otimes r}$, where $r>0$ denotes the period of $\mathrm{BS}(d,A)$. Now Theorem 5.1 immediately implies:
\begin{thm}
Let $X=\mathrm{BS}(d,A)$ be a generalized Brauer--Severi variety of period $r$ for the central simple $k$-algebra $A$ and $\mathcal{M}$ the generator of $\mathrm{Pic}(\mathrm{BS}(d,A))$. Then all indecomposable $AS$-bundles of finite rank are of the form:
\begin{center}
$\mathcal{M}^{\otimes a}\otimes \mathcal{W}_{\mathcal{L}^{\otimes j}}$
\end{center} 
with unique $a\in\mathbb{Z}$ and unique $0\leq j\leq r-1$. 
\end{thm}
An immediate consequence of Proposition 5.3 is the following: 
\begin{prop}
Let $X=\mathrm{BS}(d,A)$ be a generalized Brauer--Severi variety over $k$ for the central simple $k$-algebra $A$ of degree $n$. Then for all $j\in\mathbb{Z}$ one has:
\begin{eqnarray*}
\mathrm{rk}(\mathcal{W}_{\mathcal{L}^{\otimes j}})=\mathrm{ind}(A^{\otimes j\cdot d}).
\end{eqnarray*}
\end{prop}
\begin{proof}
Let $\lambda=(\lambda_1,\lambda_2,...,\lambda_{d})$ be a partition with $0\leq \lambda_i\leq n-d$ and $\mathcal{S}$ the tautological sheaf on $\mathrm{BS}(d,A)\otimes_k k^{sep}\simeq \mathrm{Grass}(d,n)$. As mentioned above, in [25], Section 4 it is proved that the vector bundles $\Sigma^{\lambda}(\mathcal{S})^{n\cdot |\lambda|}$ descent to vector bundles $\mathcal{N}_{\lambda}$ which are unique up to isomorphism according to Proposition 3.4. In particular, $\mathrm{det}(\mathcal{S})^{\oplus n\cdot d}$ descents to a vector bundle which we denote by $\mathcal{N}$. For the endomorphism algebra of $\mathcal{N}^{\vee}$ one has $\mathrm{End}(\mathcal{N}^{\vee})\simeq A^{\otimes d}$ (see also [25]). Since $\mathcal{N}^{\vee}$ is pure of type $\mathrm{det}(\mathcal{S}^{\vee})$, but in general not indecomposable, we conclude that $A^{\otimes d}$ is a matrix algebra over $\mathrm{End}(\mathcal{W}_{\mathcal{L}})$ (see Proposition 3.6). Proposition 5.3 gives $\mathrm{rk}(\mathcal{W}_{\mathcal{L}^{\otimes j}})=\mathrm{ind}(\mathrm{End}(\mathcal{W}_{\mathcal{L}})^{\otimes j})$. But the index of $\mathrm{End}(\mathcal{W}_{\mathcal{L}})^{\otimes j}$ is the same as the index of $A^{\otimes j\cdot d}$. This completes the proof.
\end{proof}
\begin{rema}
\textnormal{Let $G=\mathrm{Gal}(k^{sep}|k)$ be the absolute Galois group and $Y$ the generalized Brauer--Severi variety $\mathrm{BS}(d,A)$. With a result of Blanchet [10], Theorem 7, the exact sequence} 
\begin{eqnarray*}
0\longrightarrow \mathrm{Pic}(Y)\longrightarrow \mathrm{Pic}^G(Y\otimes_k k^{sep})\overset{\delta}{\longrightarrow} \mathrm{Br}(k)\overset{res}{\longrightarrow}\mathrm{Br}(F(Y))
\end{eqnarray*}
\textnormal{and the arguments of the proof of Theorem 5.4.1 of [15] one directly see that the period of $Y$ (as defined in Section 5) is equal to the period of $A^{\otimes d}$.}
\end{rema}
Now we prove Propositions 6.11 and 6.12 for generalized Brauer--Severi varieties.
\begin{prop}
Let $D$ be a central division algebra over $k$ of degree $n$ and $A=M_m(D)$. Then $\mathrm{BS}(d,D)$ and $\mathrm{BS}(d,A)$ have the same $AS$-type. 
\end{prop}
\begin{proof}
Since $D$ and $A$ are Brauer equivalent, we conclude that $\mathrm{ind}(D^{\otimes i})=\mathrm{ind}(A^{\otimes i})$ for all $i\in\mathbb{Z}$. Thus $\mathrm{ind}(D^{\otimes d\cdot i})=\mathrm{ind}(A^{\otimes d\cdot i})$ for all $i\in\mathbb{Z}$. Furthermore, the period of $D^{\otimes d}$ equals the period of $A^{\otimes d}$. Applying Proposition 7.2 and Theorem 7.1 yields the assertion.
\end{proof}
\begin{prop}
Let $X=\mathrm{BS}(d,A)$ and $Y=\mathrm{BS}(d,B)$ be two birational generalized Brauer--Severi varieties over $k$, then they have the same $AS$-type.
\end{prop}
\begin{proof}
As a consequence of [10], Theorem 7, the period of $A^{\otimes d}$ equals the period of $B^{\otimes d}$ and hence the period of $X$ equals the period of $Y$. Now the same arguments as in the proof of Proposition 6.12 show that $\mathrm{ind}(A^{\otimes d\cdot j})=\mathrm{ind}(B^{\otimes d\cdot j})$. From Theorem 7.1 and Proposition 7.2 we conclude that both have the same $AS$-type.
\end{proof}
Ottaviani [30], Theorem 2.1 proved a splitting criterion for vector bundles on $\mathrm{Grass}(d,n)$, at least if $\mathrm{char}(k)=0$. As a consequence of this result we obtain a criterion for a vector undle on $\mathrm{BS}(d,A)$ to be a $AS$-bundle. We first cite Ottaviani's result.
\begin{thm}
Let $k$ be a field of characteristic zero and $n\geq 3$. Then a locally free sheaf $\mathcal{E}$ of finite rank on $X=\mathrm{Grass}_k(d,n)$ splits as a direct sum of invertible sheaves if and only if for $0<r <\mathrm{dim}(X)$ and all $t\in\mathbb{Z}$ one has
\begin{eqnarray*}
H^r(X,\bigwedge^{i_1}(\mathcal{Q}^{\vee})\otimes\cdots\otimes\bigwedge^{i_s}(\mathcal{Q}^{\vee})\otimes\mathcal{E}(t))=0
\end{eqnarray*} for all $i_1,...,i_s$ such that $0\leq i_1,...,i_s\leq n-d$ and $s\leq d$.
\end{thm}
Now denote by $\mathcal{M}$ the (ample) generator of $\mathrm{Pic}(\mathrm{BS}(d,A))$. Considering the above $i_1,...,i_s$, we note that for a fixed $s$-tupel $i_1,...,i_s$ the $i_s$ can be ordered in a weakly decreasing way. We denote the reordering by $\lambda_1\geq \lambda_2\geq...\geq \lambda_s$. So in this way we get a partition $\lambda=(\lambda_1,...,\lambda_s)$ and we can associate a Young diagram to it with at most $d$ rows and $n-d$ columns. Now let $\mu'$ be the conjugate of the partition $\mu=(\lambda_i)$ (we have exactly $\lambda_i$ boxes). Then we have $\Sigma^{\mu'}(\mathcal{Q}^{\vee})=\bigwedge^{\lambda_i}(\mathcal{Q}^{\vee})$ on $\mathrm{BS}(d,A)\otimes_k\bar{k}\simeq \mathrm{Grass}_{\bar{k}}(d,n)$. In [25], Section 4 it is shown that $(\Sigma^{\mu'}(\mathcal{Q}^{\vee}))^{\oplus n\cdot |\mu'|}$ descents to a vector bundle $\mathcal{P}_{\lambda_i}$ on $\mathrm{BS}(d,A)$. From Proposition 3.4 we know that these vector bundles are unique up to isomorphism. We simply write $\mathcal{F}(m)$ for $\mathcal{F}\otimes \mathcal{M}^{\otimes m}$, where $\mathcal{M}\in\mathrm{Pic}(\mathrm{BS}(d,A))$ denotes the generator. With this notation we have the following result:

\begin{thm}[AS-criterion]
Let $\mathrm{BS}(d,A)$ be the generalized Brauer--Severi variety of period $r$ for the central simple $k$-algebra $A$ of degree $n\geq 3$ and $\mathcal{P}_{\lambda_i}$ the vector bundles from above. A vector bundle $\mathcal{E}$ is a $AS$-bundle if and only if for $0<r <\mathrm{dim}(\mathrm{BS}(d,A))$ and all $t\in\mathbb{Z}$ one has
\begin{eqnarray*}
H^r(\mathrm{BS}(d,A), \mathcal{P}_{\lambda_1}\otimes\cdots\otimes\mathcal{P}_{\lambda_s}\otimes \mathcal{E}(t))=0
\end{eqnarray*} for all $\lambda_1\geq \lambda_2\geq ...\geq\lambda_s$ such that $0\leq \lambda_1,...,\lambda_s\leq n-d$ and $s\leq d$.
\end{thm}
We note that Levine, Srinivas and Weyman [25] calculated the $K$-theory of generalized Brauer--Severi varieties and showed that the Grothendieck group is generated by $\mathcal{N}_{\lambda}$, where $\mathcal{N}_{\lambda}$ are the vector bundles obtained by descent from $\Sigma^{\lambda}(\mathcal{S})^{\oplus n\cdot|\lambda|}$. Therefore, as distinguished from the case of Brauer--Severi varieties, the $AS$-bundles on generalized Brauer--Severi varieties do not generate the Grothendieck group of $\mathrm{BS}(d,A)$ for $d>1$.

{\small MATHEMATISCHES INSTITUT, HEINRICH--HEINE--UNIVERSIT\"AT 40225 D\"USSELDORF, GERMANY}\\
E-mail adress: novakovic@math.uni-duesseldorf.de


\begin{thebibliography}{999}
\bibitem{teil 1} K. Akin, D.A. Buchsbaum and J. Weyman: Schur Functors and Schur Complexes. Adv. in Math. Vol. 44 (1982), 207-278.
\bibitem{teil 2} F.W. Anderson and K.R. Fuller: Rings and Categories of modules. Springer-Verlag, New York (1974).
\bibitem{teil 3} J.K. Arason, R. Elman and B. Jacob: On indecomposable vector bundles. Comm. in Alg. Vol. 20 (1992), 1323-1351.
\bibitem{teil 4} M. Artin: Brauer-Severi varieties. Brauer groups in ring theory and algebraic geometry, Lecture Notes in Math. 917, Notes by A. Verschoren, Berlin, New York: Springer-Verlag (1982), 194?10.
\bibitem{teil 5} M. Atiyah: On the Krull-Schmidt Theorem with application to sheaves. Bull. Soc. Math. France Vol. 84 (1956), 307-317.
\bibitem{teil 6} A. Auel, E. Brussel, S. Garibaldi and U. Vishne: Open Problems on Central Simple Algebras. arXiv:1006.3304 [math.RA] (2010).
\bibitem{teil 7} I. Biswas and D. Nagaraj: Classification of real algebraic vector bundles over the real anisotropic conic. Int. J. Math. Vol. 16 (2005), 1207-1220.
\bibitem{teil 8} I. Biswas and D. Nagaraj: Absolutely split real algebraic vector bundles over a real form of projective space. Bull. Sci. Math. Vol. 131 (2007), 686-696.
\bibitem{teil 9} I. Biswas and D. Nagaraj: Vector bundles over a nondegenerate conic. J. Aust. Math. Soc. Vol. 86 (2009), 145-154.
\bibitem{teil 10} A. Blanchet: Function Fields of Generalized Brauer--Severi Varieties. Communications in Algebra. Vol. 19 (1991), 97-118.
\bibitem{teil 11} N. Bourbaki: Alg\'ebre, Chapitre X, Masson. (1980).
\bibitem{teil 12} A.D. de Jong: The period-index problem for the Brauer group of an algebraic surface. Duke Math. J. 123 No. 1 (2004), 71-94.
\bibitem{teil 13} W. Fulton: Young Tableaux, with Applications to Representation Theory and Geometry. Cambridge University Press (1997).
\bibitem{teil 14} W. Fulton: Intersection Theory. Ergebnisse der Mathematik und ihre Grenzgebiete, Springer, second edition (1998).
\bibitem{teil 15} P. Gille and T. Szamuely: Central Simple Algebras and Galois Cohomology. Cambridge Studies in advanced Mathematics. 101. Cambridge University Press. (2006)
\bibitem{teil 16} A. Grothendieck: Sur la classification des fibr$\mathrm{\acute{e}}$ holomorphes sur la sph$\mathrm{\grave{e}}$re de Riemann. Amer. J. Math. Vol. 79 (1957), 121-138.
\bibitem{teil 17} A. Grothendieck: Le group de Brauer I: Algebras d Azumaya et interpretations diverses, Seminaire Bourbaki. No. 290 (1964).
\bibitem{teil 18} A. Grothendieck: Le group de Brauer II: Theorie cohomologique, Seminaire Bourbaki. No. 297 (1965).
\bibitem{teil 19} A. Grothendieck et J. Dieudonne: Elements de G\'eom\'etrie Alg\'ebrique. Publ. Math. IHES. (1960-68).
\bibitem{teil 20} A. Grothendieck: Fondements de la G\'eom\'etrie Alg\'ebrique [Extraits du S\'eminaire Bourbaki 1957-1962], Secr\'etariat math\'ematique (1962).
\bibitem{teil 21} R. Hartshorne: Algebraic Geometry. Springer-Verlag, New York, Berlin, Heidelberg (1977).
\bibitem{teil 22} N. Karpenko: Codimension 2 Cycles on Severi--Brauer Varieties. K-Theory 13 No. 4 (1998), 305-330.
\bibitem{teil 23} S. L. Kleiman: The Picard scheme: arXiv:math/0504020 [math.AG] (2005).
\bibitem{KO} 	J. Koll\'ar: Severi--Brauer varieties: A geometric treatment. 
\bibitem{teil 24} M. Levine, V. Srinivas and J. Weyman: K-Theory of Twisted Grassmannians. K-Theory Vol. 3 (1989), 99-121.
\bibitem{teil 25} J. Milne: $\mathrm{\acute{E}}$tale Cohomology. Princeton Mathematical Series. Vol. 33, Princeton University Press (1980).
\bibitem{teil 26} S. Novakovi$\mathrm{\acute{c}}$: Der Spaltungssatz f\"ur Quadriken. Diploma Thesis, Heinrich-Heine Universit\"at D\"usseldorf (2009).
\bibitem{teil 27} S. Novakovi$\mathrm{\acute{c}}$: Absolutely split locally free sheaves on Brauer-Severi varieties of index two. Bull. d. Sci. Math. Vol. 136 (2012), 413-422.
\bibitem{teil 28} C. Okonek, M. Schneider and H. Spindler: Vector Bundles on Complex Projective Space. Progress in Mathematics. 3. Birkh\"auser, Boston (1980).
\bibitem{teil 29} G. Ottaviani: Some extensions of the Horrocks criterion to vector bundles on grassmannians and quadrics. Annali di Matem. Vol. 155 (1989), 317-341.
\bibitem{teil 30} S. Pumpl\"un: Vector bundles and symmetric bilinear forms over curves of genus 1 and arbitrary index. Mathematische Zeitschrift Vol. 246 (2004), 563-602
\bibitem{teil 31} D. Quillen: Higher algebraic K-theory. Algebraic K-theory I, Lecture Notes in Math. 341, Springer (1979), 85-147.
\bibitem{teil 32} D. Saltman: Lectures on division algebras. Vol. 94 of CBMS Regional conference Series in Mathematics. American Math. Soc. (1999).
\bibitem{teil 33} J.-P. Serre: Local Fields. Springer-Verlag, New York, Berlin (1980).
\bibitem{teil 34} J.-P. Serre: Cohomologie Galoisienne. Lecure Notes in Mathematics 5, Springer-Verlag Berlin (1994); English translation: Galois Cohomology, Springer-Verlag Berlin (2002).
\bibitem{teil 35} V.E. Voskresenskii: Algebraic Groups and Their Birational Invariants. Translated by Boris Kunyavski, Translations of Math. Monographs AMS Vol. 179 (2000).
\bibitem{teil 36} R. Wiegand: Torsion in Picard Groups of Affine Rings. Contemp. Math. Vol 159 (1994), 433-444.

\end{thebibliography}
\end{document}